\documentclass[journal]{IEEEtran}
\ifCLASSINFOpdf
\else
\fi
\usepackage{graphics} 
\usepackage{epsfig} 
\usepackage{amsmath} 
\usepackage{cases}
\usepackage{amssymb}  
\usepackage{verbatim}
\usepackage{subfigure}
\usepackage{cite}
\usepackage{url}
\usepackage[ruled,linesnumbered]{algorithm2e}
\usepackage{tikz,times}
\usepackage{endnotes}
\usepackage{booktabs} 
\usepackage{threeparttable}
\usepackage{multirow}
\usepackage[top=2cm, bottom=2cm, left=2cm, right=2cm]{geometry}

\usepackage{amsthm}
\newtheorem{theorem}{\textbf{Theorem}}
\newtheorem{lemma}{\textbf{Lemma}}
\newtheorem{assumption}{Assumption}

\newtheorem{remark}{Remark}
\newtheorem{definition}{Definition}

\newtheorem{corollary}{Corollary}
\newcommand{\tabincell}[2]{\begin{tabular}{@{}#1@{}}#2\end{tabular}}

\begin{document}
\title{Heuristic Learning for Co-Design Scheme of Optimal Sequential Attack }
\author{Xiaoyu Luo$^\dag$, Haoxuan Pan$^\dag$, Chongrong Fang$^\dag$, Chengcheng Zhao$^\ddag$, Peng Cheng$^\ddag$, and Jianping He$^\dag$
\thanks{$^\dag$: The Department of Automation, Shanghai Jiao Tong University, and Key Laboratory of System Control and Information Processing, Ministry of Education of China, Shanghai 200240, China. E-mail: xyl.sjtu@sjtu.edu.cn, panhaoxuan@sjtu.edu.cn, crfang@sjtu.edu.cn, jphe@sjtu.edu.cn.}
\thanks{$^\ddag$: The State Key Laboratory of Industrial Control Technology and Institute of Cyberspace Research, Zhejiang University, China. E-mail: chengchengzhao@zju.edu.cn, lunarheart@zju.edu.cn.}
}

\maketitle

{\color{black}
\begin{abstract}
This paper considers a novel co-design problem of the optimal \textit{sequential} attack, whose attack strategy changes with the time series, and in which the \textit{sequential} attack selection strategy and \textit{sequential} attack signal are simultaneously designed. Different from the existing attack design works that separately focus on attack subsets or attack signals, the joint design of the attack strategy poses a huge challenge due to the deep coupling relation between the \textit{sequential} attack selection strategy and \textit{sequential} attack signal. In this manuscript, we decompose the sequential co-design problem into two equivalent sub-problems. Specifically, we first derive an analytical closed-form expression between the optimal attack signal and the sequential attack selection strategy. Furthermore, we prove the finite-time inverse convergence of the critical parameters in the injected optimal attack signal by discrete-time Lyapunov analysis, which enables the efficient off-line design of the attack signal and saves computing resources. Finally, we exploit its relationship to design a heuristic two-stage learning-based joint attack algorithm (HTL-JA), which can accelerate realization of the attack target compared to the one-stage proximal-policy-optimization-based (PPO) algorithm. Extensive simulations are conducted to show the effectiveness of the injected optimal sequential attack.

\end{abstract}}

\begin{IEEEkeywords}
	False data injection attacks, learning-based methods, attack selection strategy, convergence
\end{IEEEkeywords}
\section{INTRODUCTION}\label{I}
\IEEEPARstart{S}{ecurity} issues are becoming increasingly prominent in networked control systems (NCSs) as network technologies are extensively used to connect physical components within a control loop \cite{zhang2019networked}. In NCSs, false data injection (FDI) --- whereby an adversary injects false data by manipulating sensor readings or communication channels --- is a commonly encountered form of cyber attack \cite{luo2022feedback}. Crucially, through an FDI attack, an adversary can cause significant damage to control components while remaining undetected. For instance, on June 27, 2022, anonymous hacker organization Gonjeshke Darande carried out cyber attacks against Iran's steel industry such that a heavy machine on a billet production line broke down and caused a fire \cite{Gonjeshke2022attack}. As a result, the steel industry had to halt production, leading to lots of economic losses.

\subsection{Motivations}
Considerable efforts have been devoted to studying the effects of potential FDI attacks \cite{liu2011false,mo2010false,sui2020vulnerability,zhu2014resilience,tan2017modeling} and designing the optimal FDI attack strategies \cite{chen2017cyber,li2019optimal,jafari2022optimal}. For instance, Chen \emph{et al.} \cite{chen2017cyber} found an optimal attack strategy to balance the control objective and the detection avoidance objective.
Li \emph{et al.} derived the optimal linear attack vector injected in the sensor readings to degrade the system estimation performance \cite{li2019optimal}. Jafari \emph{et al.} \cite{jafari2022optimal} studied an optimal false data injection attack (OFDIA) on automatic generation control (AGC) in power systems to destroy the frequency stability. 
Most of these works focus on the design of the injected optimal attack signal to meet the given objective function. Besides, there are some researchers aiming at developing the FDI attack selection strategy \cite{wu2018optimal,wu2019optimal,ye2020complexity,luo2022submodularity}. Wu \emph{et al.} solved an optimal switching data injection attack design problem where only one actuator is compromised each time to minimize the quadratic cost function \cite{wu2018optimal}. In \cite{luo2022submodularity}, the adversary with limited capability aims to select a subset of agents and manipulate their local multi-dimensional states to maximize the consensus convergence error by utilizing the submodularity optimization theory. It shows distinct attack effects under different attack selection strategies. 

{\color{black}Note that there exist three interesting problems worthy of further investigation. The first one is to explore the relationship between the injected attack signal and the attack selection strategy. For an adversary, selecting which agent to compromise and how much attack signal to inject are two key tasks. Usually, they are coupled and integrated into the system. 
It is significant to build an analytic expression for both and analyze how the attack selection strategy influences the injected attack signal. With this relationship, it is beneficial to probe the adversary's potential capability and predict its possible behavior.
It is worth noting that few works focus on excavating the analytic relationship between them.  
The second one is to excavate the convergence property of the injected optimal attack signal. It is intriguing and promising to demonstrate the characteristic of the injected attack signal. Once its convergence property is excavated, the adversary can effectively inject attack signal and save unnecessary computing resources to maximize the malicious effects. 
The third one is to tackle the sequential attack selection problem as time varies instead of a fixed compromised subset. It is more practical for an intelligent adversary to maximize the attack effects with time-varying attack selection strategies. For example, in smart grids, many substations can be compromised sequentially, whose combinations are prone to cause severe large-scale blackouts \cite{zhu2014resilience}.
Additionally, from the perspective of the system protection, it is advantageous to analyze the potential system's vulnerability and design resilient algorithms to improve the system's security.}

\subsection{Contributions}
In this paper, we study the co-design problem of optimal sequential FDI attack where the adversary aims to specify how to select the compromised agent and inject the attack signal sequentially. Concretely, we derive the relationship between the injected optimal sequential attack signal and the attack selection strategy. Meanwhile, we desire to seek the potential convergence property of the injected attack signal where the adversary aims to steer the system state value to an expected malicious one in a discrete-time system. Compared to our conference version \cite{luo2023optimal}, we extend the sequential attack signal design problem to the sequential attack selection problem and propose a heuristic two-stage learning-based joint attack algorithm (HTL-JA) to reach optimal performance. Moreover, we significantly enrich the related works, motivation and simulation results.
The main contributions are summarized as follows.

\begin{itemize}
	\item We construct a sequential joint attack design framework where
	      the adversary selects the sequential attack subsets and injects sequential attack signal over sampling times to drive the system state to a desired malicious one.
	\item We derive an analytical closed-form expression between the optimal sequential attack signal and the attack selection strategy, in which they are
	      deeply coupled. Moreover, we theoretically characterize the finite-time inverse convergence of the critical parameters in the obtained optimal sequential attack signal via the discrete-time Lyapunov analysis.
	\item We propose a heuristic two-stage learning-based joint attack algorithm (HTL-JA), which contributes to acquiring the sequential attack subset and speeding up the realization of the attack target.
\end{itemize}
\subsection{Paper Organization}
The rest of the paper is organized as follows. 
Related works are reviewed in Section \ref{related-work}. 
Section \ref{II} introduces the system model and the adversary model, and formulates the FDI attack co-design problem. 
In Section \ref{III}, the optimal sequential attack signal is designed.
Section \ref{attack-selection} proposes a heuristic two-stage learning-based attack selection strategy. Simulation results are presented in Section \ref{IV}. Finally, we conclude our work in Section \ref{V}.\\

\textbf{Notations.} Let $\mathbb{R}$ denote the set of real numbers. For a vector $l_1 \in \mathbb{R}^{p}$, we have $\|l_1\|^{2}_R \triangleq l_1^{\mathrm{T}} R l_1$ with the positive definite weight matrix $R$. We denote $I_n$ and $1_n$ as the $n$-dimensional diagonal unit matrix and column vector with all elements of $1$, respectively. For a matrix $L_1$, we let $L_1^{*}$ denote its Hermitian matrix.
\section{Related Works}\label{related-work}
A great deal of literature on the design of FDI attacks can be roughly divided into two categories, including designing the false data injection (FDI) attack signal \cite{liu2011false,Zhang2020optimal,Lu2022false,luo2022submodularity,luo2022model,kim2014subspace,an2017data,zhao2022data} and the attack selection strategy \cite{ wu2018optimal,wu2019optimal,ye2020complexity,luo2022submodularity}.

\textbf{FDI Attack Signal Design:}
The first fundamental work on launching the FDI attack signal is proposed by Liu \emph{et al.} \cite{liu2011false} in which the adversary could compromise measurements and change the results of state estimation without being detected by the bad measurement detection technique in smart grids. It reveals the potential secure breach of the power system when offline observations and system information are available to the adversary. Note that the design of the FDI attack signal depends on its attack objective and available information about the system model.
In terms of the attack objective, the attack signal can be designed to maximize the remote state estimation error \cite{Lu2022false}, the consensus error \cite{luo2022submodularity}, the tracking error \cite{Mousavinejad2021resilient}, and the quadratic cost function \cite{Zhang2020optimal}, to name a few.
In view of the available information, the attack signal can be divided into two types, including model-based attack signal \cite{wang2021optimal,Lu2022false} and data-driven one \cite{kim2014subspace,an2017data,zhao2022data}. Among them, there are a large number of works on the optimal attack signal design. Nevertheless, few works focus on the property analysis of the optimal attack signal and explore the characteristic of the attack signal. 

\textbf{FDI Attack Selection Strategy:}
In \cite{Pasqualetti2013attack}, Pasqualetti \emph{et al.} first studied the undetectable and unidentifiable FDI attack set where the adversary knows the full information about the system model and compromises sensors and actuators. It shows that the adversary has the ability to manipulate multiple attack objects without being detected.
In the following, we review only works that are most pertinent to ours \cite{wu2018optimal,wu2019optimal,ye2020complexity,luo2022submodularity}. One type of work is to design switching attacks where the adversary can compromise only one agent at a time, which is basically considered as a kind of attack selection strategy. For example, Wu \emph{et al.} \cite{wu2019optimal} formulated an optimal switching attack design problem where the adversary aims to maximize the quadratic cost of states by determining the optimal compromised sensor sets. To relax the limitation on the number of compromised agents at a time, Luo \emph{et al.} \cite{luo2022submodularity} proposed a submodularity-based FDI attack selection scheme where the adversary can manipulate multi-dimensional states for multiple agents. Nevertheless, note that the attack selection strategy is fixed and time-invariant. In practical scenarios, an intelligent adversary has the ability to change the subset of the compromised agents and dynamically adjust its attack selection strategy. For example, the substations can be compromised sequentially, whose combinations can cause severe large-scale outages in smart grids \cite{zhu2014resilience}. 
Hence, from the perspective of the adversary, it is promising and interesting to seek an efficient method to obtain a sequential attack selection strategy where the attack subset varies as the sequential sampling time, which is more practical and has better attack effects than the time-invariant attack selection strategy.

In a nutshell, different from separately handling the design of the attack signal and attack selection scheme,
our work mainly centers on constructing the bridge between the attack signal and the attack selection strategy and tackling a sequential FDI attack co-design problem with these two coupled variables.
\section{PROBLEM FORMULATION}\label{II}

\subsection{System Model and Adversary Model}
Consider a discrete-time dynamical system
\begin{align}\label{eq1}
	x_{k+1}=A_k x_k +B_k u_k,
\end{align}
where $A_k\in \mathbb{R}^{n\times n}$, $B_k \in \mathbb{R}^{n\times m}$ are the system matrices, $x_k\in \mathbb{R}^{n}$ and $u_k\in \mathbb{R}^{m}$ are the system state and system input at time $k$, respectively. We set the linear feedback controller as $u_k=L_k x_k$. Then, we have 
\begin{align}\label{addeq1}
	x_{k+1}=W_k x_k,
\end{align}
with the system matrix $W_k=A_k + B_k L_k$. 

{\color{black}Consider an adversary can compromise system \eqref{eq1} by altering the original control law $u_k$ or deviating the control signals from the true values, thus indirectly manipulating the system states $x_k$. For an adversary, it has the ability to flexibly select which agent to tamper with and design the injected attack signal simultaneously as time $k$ varies.
The dynamic system \eqref{addeq1} under such attack can be remodeled as
\begin{align}\label{eq2}
	x_{k+1}^{a}=W_k x^{a}_k + \Gamma_k \theta_k,
\end{align}
where $\theta_k\in \mathbb{R}$ is called \textit{sequential} attack signal, the \textit{sequential} attack selection strategy $\Gamma_k=[\gamma_k^1,\ldots,\gamma_k^n]^{\mathrm{T}}\in \mathbb{R}^{n}$ with the binary variable $\gamma_k^i=1$ if the $i$-th agent is compromised at time $k$ and $\gamma_k^i=0$ otherwise.
Then, we make the following assumption about the ability of the adversary and the definition of a sequential attack.}
\begin{assumption}\label{ass2}
	The adversary knows the exact knowledge of the system model.
\end{assumption}

Assumption \ref{ass2} is a common and implicit condition for the adversary to inject false data successfully \cite{guo2018worst}.
\begin{definition}
	\textit{(Sequential attack)} An attack is called sequential if it launches attack strategies (selects attack subsets or injects attack signals) as the sampling time $k$ varies.
\end{definition}


{\color{black}\subsection{Problem Formulation}
In this work, we consider that the adversary's objective is to steer the system state to the expected malicious one as closely as possible in finite time by injecting the false data $\Gamma_k \theta_k$. The sum of the state error is characterized by $J_1$, i.e.,
\vspace{-0.2cm}
\begin{align*}
    J_1=\sum_{k=1}^{N} \left(\|x^a_k-x^*\|^2_{P_k}\right)+ \|x^{a}_{N+1}-x^*\|^2_H,
\end{align*}
where $N$ is the given upper bound of finite-time iteration, $x^{*}$ is the expected malicious state predefined by the adversary, and $P_k$ and $H$ are the positive definite weight matrices.

We also consider that the adversary desires to save the attack energy. The energy of injected false data is denoted as $J_2$, i.e.,
\vspace{-0.1cm}
\begin{align*}
    J_2=\sum_{k=0}^{N}\left(\|\Gamma_k \theta_k\|^2_{Q_k}\right),
\end{align*}
where $Q_k$ is the positive definite weight matrix.

Therefore, the total goal of the adversary is to reduce both the state error between the true system state and the expected malicious one and the consumed attack energy as much as possible. In $N$ iterations, the injected false data includes the \textit{sequential} attack signal $\theta \triangleq \{\theta_0,\theta_1,\ldots,\theta_N\}$ and \textit{sequential} attack selection strategy $\Gamma \triangleq \{ \Gamma_0, \Gamma_1, \ldots, \Gamma_N\}$. Under the injected false data, the sum of the state error and the consumed attack energy is expected to be minimized. To this end, constrained by the intrinsic system dynamic model \eqref{eq2}, we construct the following optimization problem $\mathcal{P}_0$.
\begin{align}\label{problem_0}
	\mathbf{\mathcal{P}_0}: \quad & \mathrm{min}_{\{\theta,\Gamma\}} ~J= J_1 + J_2 \\
	                              & \mathrm{s.t.}~ x_{k+1}^{a}=W_k x^{a}_k + \Gamma_k \theta_k. \nonumber
\end{align}

\subsection{Problem Decomposition}
The challenges of directly solving problem $\mathcal{P}_0$ result from the nonlinearity and non-convexity of the objective function $J$ with respect to two closely coupled optimization variables $\Gamma_k$ and $\theta_k$.
Furthermore, it is difficult to directly obtain the gradients of the objective function for variables $\theta$ and $\Gamma$ to solve problem $\mathcal{P}_0$.
If we can explore the relationship between the attack signal $\theta_k$ and the attack selection strategy $\Gamma_k$ and derive an analytical closed-form relation, it is vital to simplify the solution of problem $\mathcal{P}_0$ and reduce the difficulty of solving problem $\mathcal{P}_0$. Concretely, we could first obtain the optimal attack signal when the attack selection strategy is given and known. Then we explore the feasible attack selection strategy based on the optimal attack signal in which there exists the relationship between the attack signal and the attack selection strategy. 
Thus, we decompose $\mathcal{P}_0$ into the following two sub-problems, i.e., problem $\mathcal{P}_1$ and $\mathcal{P}_2$.
\vspace{-0.1cm}
\begin{align}\label{problem_1}
	\mathbf{\mathcal{P}_1}: \quad & \mathrm{min}_{\{\theta_0,\theta_1,\ldots,\theta_N\}} ~J= J_1 + J_2 \\
	                              & \mathrm{s.t.}~ x_{k+1}^{a}=W_k x^{a}_k + \Gamma_k \theta_k, \nonumber\\
                               & \quad ~~ \Gamma=\Gamma^{\#},\nonumber
\end{align}
where the attack selection strategy $\Gamma^{\#}$ is given and known. 
\vspace{-0.1cm}
\begin{flalign}\label{problem_2}
	\mathbf{\mathcal{P}_2}: \quad & \mathrm{min}_{\{\Gamma_0,\Gamma_1,\ldots,\Gamma_N\}} ~J= J_1 + J_2          \\
	                              & \mathrm{s.t.}~ x_{k+1}^{a}=W_k x^{a}_k + \Gamma_k \theta_k, \nonumber       \\
	                              & \quad ~~\theta_k = \theta_k^{\#} \nonumber,
\end{flalign}
where $\theta_k^{\#}$ is obtained by solving problem $\mathcal{P}_1$. The two sub-problems after decomposition are equivalent to the original problem since the multivariate optimization problem can be reduced to multiple univariate optimization problems if the closed-form analytical relationships among them are known. 

To address problem $\mathcal{P}_0$, we first focus on problem $\mathcal{P}_1$. In problem $\mathcal{P}_1$, we mainly analyze the relationship between the injected attack signal $\theta$ and the attack selection strategy $\Gamma$ based on dynamic programming. Later, in Section \ref{attack-selection}, we will deal with $\Gamma$ with the heuristic learning-based algorithm under the obtained optimal \textit{sequential} attack signal $\theta$.}

\section{Optimal Sequential Attack Signal Scheme}\label{III}
In this section, we solve problem $\mathcal{P}_1$ and derive the optimal sequential attack signal based on dynamic programming. Then, we excavate its critical parameters' property. 

\subsection{Optimal Attack Signal Design}
{\color{black}Before demonstrating the sequential attack signal scheme, we introduce the notion of dynamic programming. Based on the Bellman principle of optimality, the key idea of dynamic programming is to transform the multi-stage decision problem to multiple single-stage decision problems. Concretely, if a multi-stage decision process satisfies the principle of optimality, it means that the decision sequence of the subsequent stages must be optimal for a state caused by the previous decision regardless of the initial state and the initial decision. } 

Motivated by the above observations, we aim to deal with problem $\mathcal{P}_1$ with dynamic programming. The schematic of the optimal sequential attack signal design is shown in Fig. \ref{schematic}, where the multi-time attack signal injection problem is transformed into multiple single-time attack signal injection problems. Specifically, given the attack selection strategy $\Gamma_k$, the critical parameters $F_k$ and $M_k$ can be obtained backward offline based on \eqref{F_k} and \eqref{M_k}. 
Then, the solution of problem $\mathcal{P}_1$, i.e., the optimal sequential attack signal $\theta_k$ is derived with the obtained $F_k$ and $M_k$. The detailed solution is shown in the following theorem.

\begin{figure}[t]
	\centering
	\includegraphics[width=0.45\textwidth]{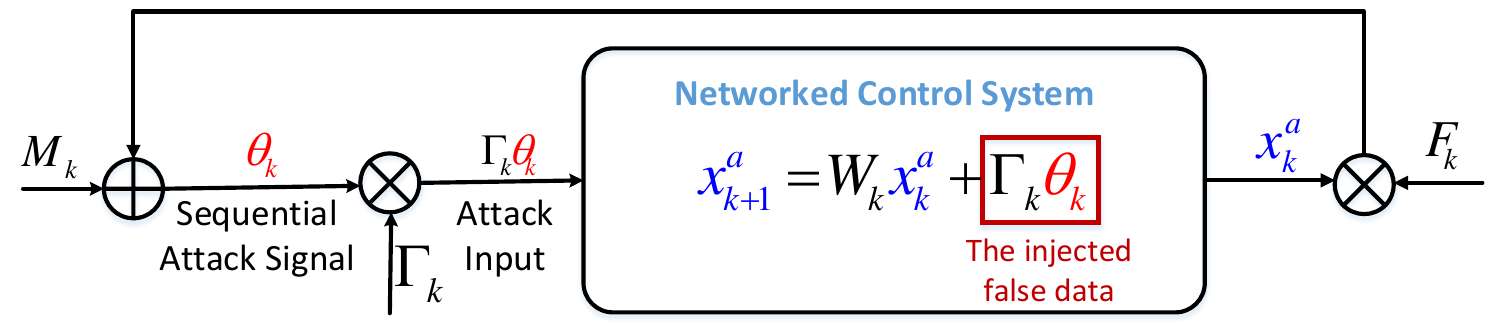}
	\caption{The schematic of the optimal sequential attack signal design}
	\label{schematic}
	\vspace*{-10pt}
\end{figure}

\begin{theorem}\label{th:optimal}
	(Optimal Sequential Attack Signal) The optimal sequential attack signal $\theta_{k}$ for $k=0,1,\cdots,N$, that minimizes $J$ in \eqref{problem_1} is
	\begin{align}\label{theta_without}
		\theta_k= F_k x^{a}_k + M_k,
	\end{align}
	where the critical parameter
	\begin{align}\label{F_k}
		F_k=-R_k^{-1}\Gamma_k^{\mathrm{T}} K_{k+1} W_k,
	\end{align}
	with $R_k=\Gamma_k^{\mathrm{T}}(Q_k+K_{k+1})\Gamma_k$ and the intermediate variable
	\begin{align}\label{K}
		K_k & = P_k + W_k^{\mathrm{T}} K_{k+1} W_k + 2 W_k^{\mathrm{T}} K_{k+1} \Gamma_k F_k \\ \nonumber
		    & ~~~+ F_k^{\mathrm{T}} \Gamma_k^{\mathrm{T}} (Q_k + K_{k+1} ) \Gamma_k F_k,
	\end{align}and another critical parameter $M_k=$
	{\small{
		\begin{align}\label{M_k}
			\begin{split}
				\left\{
				\begin{array}{l}
					R_k^{-1} \Gamma_k^{\mathrm{T}} K_{k+1} x^{*}, k=N, \\
					R_k^{-1} \Gamma_k^{\mathrm{T}} K_{k+1} (P_{k+1} x^{*} - W_{k+1}^{\mathrm{T}} K_{k+2} \Gamma_{k+1} M_{k+1}), k\neq N.
				\end{array}
				\right.
			\end{split}
		\end{align}}}
\end{theorem}

\begin{proof}
The proof can be completed by solving the Bellman equation backward from time $N+1$ of termination, shown in Appendix \ref{APP-A}.
\end{proof}

\begin{figure}[t]
	\centering
	\includegraphics[width=0.45\textwidth]{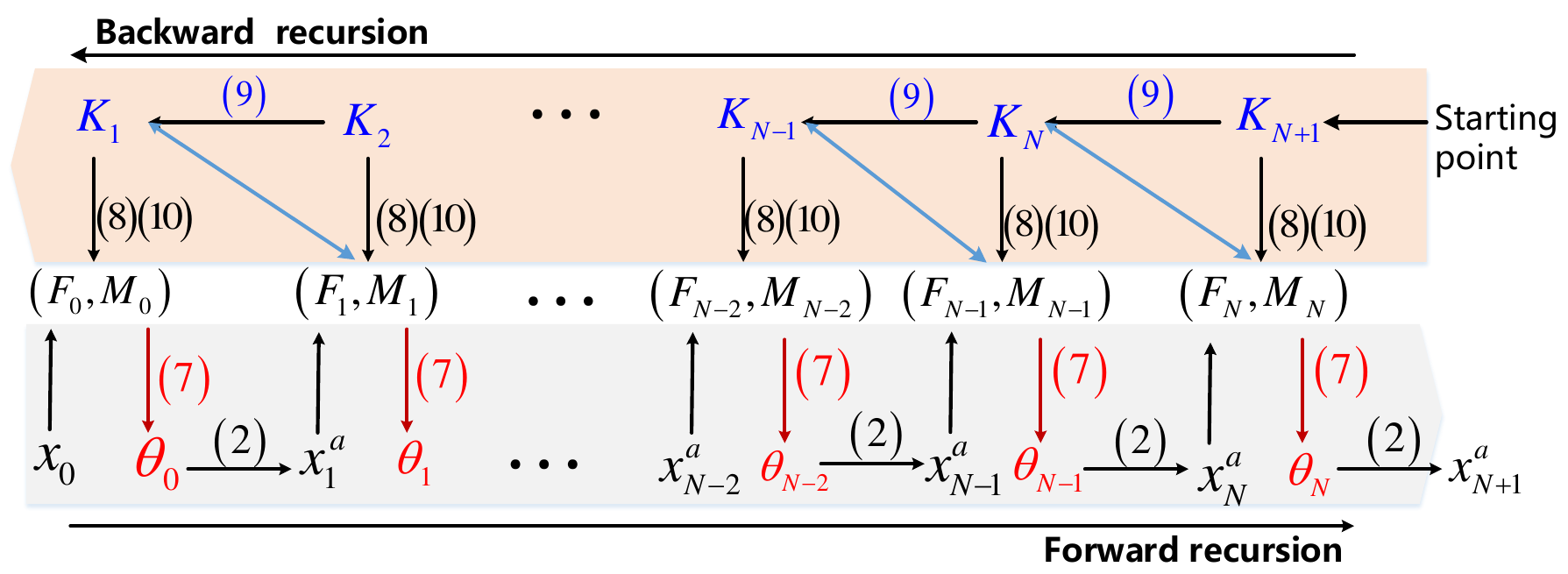}
	\caption{The recursion flow of $F_k$, $K_k$, $M_k$ and $\theta_k$}
        \label{fig:flow}
	\vspace*{-10pt}
\end{figure}
Theorem \ref{th:optimal} reveals the strongly coupled relationship between the optimal sequential attack signal and the attack selection strategy. 
Especially, the optimal attack signal $\theta_k$ at time $k$ is the function of the system state $x_k^{a}$. 
Besides, it is related to the system structure $W_k$, the expected malicious state $x^{*}$, and the initial states $x_0$. 
In other words, once the adversary knows the initial state $x_0$ and the system structure $W_k$, the optimal sequential attack signal $\theta_k$ can be designed after the adversary determines the expected malicious state $x^{*}$, the attack selection strategy $\Gamma_k$, and weight matrices $P_k$, $Q_k$ and $H$. 
As shown in Fig. \ref{fig:flow}, with the initial matrix $K_{N+1}$, $F_k$, $K_k$ and $M_k$ are derived backward based on \eqref{F_k}, \eqref{K} and \eqref{M_k}, respectively. 
Then, with the known initial states $x_0$ and \eqref{theta_without}, the adversary can directly inject optimal sequential attack signal $\theta_k$ along the forward iteration timeline.

\subsection{Property Analysis}
In this part, we demonstrate the inverse convergence of the critical parameter matrix $K_k$ in \eqref{K} and vector $F_k$ in \eqref{F_k}, respectively. Since $K_k$ and $F_k$ are derived backward, its inverse convergence is defined as follows.
\begin{definition}\label{d2}
	\textit{(Inverse Convergence)} Matrix/Vector/Point convergence is called inverse convergence if the matrix/vector/point is derived backward and converges along the reverse order of iteration time.
\end{definition}

Based on Definition \ref{d2}, we find that the sequential \{$K_{N}$, $K_{N-1}$, $K_{N-2}$, $\ldots$, $K_{2}$, $K_{1}$\} and \{$F_{N}$, $F_{N-1}$, $F_{N-2}$, $\ldots$, $F_{1}$, $F_{0}$\} converge forward, which are also called inverse convergence of $K_k$ and $F_k$. With this property, it is possible to quickly obtain the steady-state parameters $K_k$ and $F_k$. In other words, only a small number of iteration times $k$ are required to derive $F_k$ and $K_k$ backward regardless of the finite-time $N$. Based on these few backward recursions, the optimal sequential attack signal can be directly designed.

In what follows, we first analyze the symmetry and positive definiteness of $K_k$, and the system's finite-time stability, which is beneficial to proving its inverse convergence.

\begin{lemma}[Symmetry and positive definiteness of $K_k$]\label{l1}
	The matrix $K_k$ in \eqref{K} is a positive definite Hermitian matrix for $k=0,1,\ldots,N$, i.e., $K_k=K^{*}_{k} \succ 0$.
\end{lemma}
\begin{proof}
 Please see Appendix \ref{B}.
\end{proof}
\begin{corollary}\label{c1}
	$K_k$ in \eqref{K} can be simplified as
	\begin{align*}
		K_k=P_k+ W_k^{\mathrm{T}} K_{k+1} W_k -R_k^{-1} W_k^{\mathrm{T}} K_{k+1} \Gamma_k \Gamma_k^{\mathrm{T}} K_{k+1} W_k.
	\end{align*}
\end{corollary}
Combined with \eqref{K} and \eqref{F_0}, the proof is completed.

\begin{lemma}[Finite-time stability]\label{l3}
	Consider a discrete-time system with a corresponding positive definite matrix-valued Lyapunov function $\tilde{V}: \mathbb{R}^{n\times n} \rightarrow \mathbb{R}$ and let $\tilde{V}_k=\tilde{V}(K_k-K^{\star})$. Let $\alpha$ and $\epsilon$ be a constant in the open interval $(0,1)$. Let $\tilde{V}_N>0$ be the finite initial value of the Lyapunov function with respect to $K_k$. Denote $.\varphi_k \triangleq \varphi(\tilde{V}_k^{1-\alpha})$ where $\varphi: \mathbb{R}^{+}\rightarrow \mathbb{R}^{+}$ is a class-$\mathcal{K}$ function of $\tilde{V}_k^{1-\alpha}$ that satisfies
	\begin{align}\label{range}
		\frac{\varphi_k}{\varphi_N} \geq 1-\epsilon \quad \text{for} \quad \tilde{V}_k^{1-\alpha}\in (\tilde{V}_N^{1-\alpha}-\chi, \tilde{V}_N^{1-\alpha})
	\end{align}
	for some finite positive constant $\chi < \tilde{V}_N^{1-\alpha}$.
	Then, if $\tilde{V}_k$ satisfies the relation
	\begin{align}\label{eq:23}
		\tilde{V}_{k-1}-\tilde{V}_{k}=-\varphi_k \tilde{V}_k^{\alpha},
	\end{align}
	matrix $K_k$ has the steady state and converges to $K^{\star}$ for $0\leq k< \xi^{\star}$ where the positive integer $\xi^{\star}$ satisfies \eqref{ep}.
\end{lemma}
\begin{proof}
The proof could be founded in Appendix \ref{C}.
\end{proof}

Lemma \ref{l3} provides a new insight to prove the finite-time inverse convergence for the matrix $K_k$ in \eqref{K}, which is also an extension of finite-time vector forward convergence \cite{hamrah2019discrete} to matrix inverse convergence. Based on Lemma \ref{l3}, then we develop a matrix-valued Lyapunov function in the following theorem to show the inverse convergence of $K_k$.
\begin{theorem}[Finite-time inverse convergence of $K_k$]\label{th:convergence}
	Let $\xi^{\star}$ be the smallest integer for the inverse convergence of matrix $K_k$. The parameter matrix $K_k$ in \eqref{K} converges inversely when $0\leq k< \xi^{\star}$ where $\xi^{\star}$ satisfies \eqref{ep}.
\end{theorem}
\begin{proof}
The proof could be founded in Appendix \ref{D}.
\end{proof}
\begin{corollary}[Inverse Convergence of $F_k$]\label{c2}
	When the system structure $W_k$ is fixed, the parameter vector $F_k$ in \eqref{F_k} converges inversely when $0\leq k< \xi^{\star}+1$.
\end{corollary}
\begin{proof}
	Since $F_k=-R_k^{-1}\Gamma_k^{\mathrm{T}} K_{k+1} W_k$ with $R_k=\Gamma_k^{\mathrm{T}}(Q_k+K_{k+1})\Gamma_k$ and $K_k$ in \eqref{K}, the proof can be completed if the convergence of $K_{k+1}$ is guaranteed. When $0\leq k< \xi^{\star}$, $K_k$ converges inversely. Thus, $F_k$ converges when $0\leq k< \xi^{\star}+1$. The proof is completed.
\end{proof}




\begin{figure}[t]
	\centering
	\includegraphics[width=0.45\textwidth]{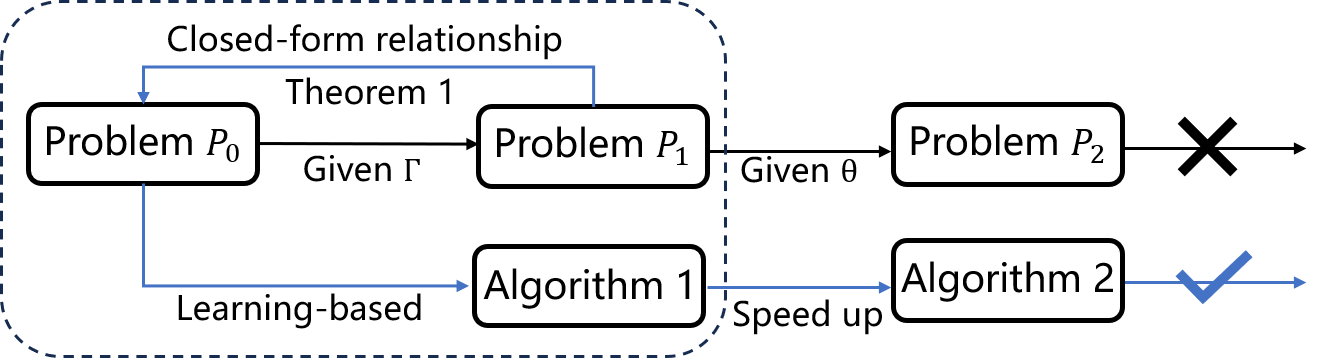}
	\caption{The solving process of problem $\mathcal{P}_0$}
	\label{solving}
	\vspace*{-10pt}
\end{figure}

\section{Learning-based Joint Attack Strategy}\label{attack-selection}
In this section, we propose a heuristic two-stage learning-based joint attack
 algorithm for solving problem $\mathcal{P}_0$, whose idea is shown in Fig. \ref{solving}. Concretely, we first obtain the optimal sequential attack signal when the attack selection strategy $\Gamma$ is given, which is derived in Theorem \ref{th:optimal} by addressing problem $\mathcal{P}_1$. Note that computing $\theta_k$ necessitates knowledge of $\Gamma_{k}$, $\Gamma_{k+1}$, $\ldots$, $\Gamma_{N}$ from Theorem \ref{th:optimal}, i.e., 
\begin{align}\label{eq:f}
	\theta_k = f(\Gamma_k,\Gamma_{k+1}, \ldots, \Gamma_{N}),
\end{align}
where $f$ is a function of $\Gamma$ and characterizes the closed-form relation between the injected attack signal and the attack selection strategy, which is obtained by solving problem $\mathcal{P}_1$.
Then, with the prior relationship between $\Gamma$ and $\theta$ in Theorem \ref{th:optimal}, we desire to deal with problem $\mathcal{P}_2$. In essence, problem $\mathcal{P}_2$ is a multi-stage decision problem with $0$-$1$ integer variables. The challenges of tackling problem $\mathcal{P}_2$ come from the $0$-$1$ integer variables. It is difficult to directly obtain the analytical optimal solution of such a multi-stage $0$-$1$ integer programming problem.
Moreover, the size of the optimal sequential attack signal $\theta_k$ depends on the future attack selection strategy $\Gamma_k$, $\Gamma_{k+1}$, $\Gamma_{k+2}$, $\ldots$, $\Gamma_{N}$, as shown in \eqref{theta_without}, making the problem more complicated to solve.
A brute-force approach to solving this problem would result in an intractable exponential time complexity of $O(n^N)$.

To address the challenge of high time complexity, we employ a heuristic algorithm.
Specifically, we use Reinforcement Learning (RL) approach, which shows great potential for handling optimization problems with sequential decision variables by leveraging the chronological information provided by a Markov Decision Process (MDP) approach.
It is worth noting that for a problem to be cast as an MDP, it must satisfy the Markov condition, which stipulates that the transition and reward are contingent \textit{solely} on the current situation~\cite{puterman1990markov}, rather than on past or future states.
However, the optimal sequential attack signal $\theta_k$ in \eqref{theta_without} cannot be directly derived based on the given attack selection strategy $\Gamma_k$, i.e., \textit{problem $\mathcal{P}_2$ cannot be modeled as an MDP process}. How to tackle the dependency relationship between $\theta_k$ and $\Gamma_k,\Gamma_{K+1},\ldots,\Gamma_{N}$ is critical for designing a feasible scheme to obtain the multi-stage attack selection strategy.

To circumvent this issue and ensure the Markov property is maintained, we recall problem $\mathcal{P}_0$, in which we add the constraint \eqref{eq:f} and the speciality is to exploit the actual optimal attack signal $\theta^*_k$ derived in \eqref{theta_without} as the prior knowledge and introduce a penalty term in the computation of the reward function. The solving procedure for this problem $\mathcal{P}_0$ will be elaborated in the subsequent subsections.  
\begin{align}\label{problem_3}
	\mathbf{\mathcal{P}_0}: \quad & \mathrm{min}_{\{\Gamma, \theta\}} ~J= J_1 + J_2 \\
	                              & \mathrm{s.t.}~ x_{k+1}^{a}=W_k x^{a}_k + \Gamma_k \theta_k, \nonumber                                   \\
	                              & \quad~~\theta^*_k= f(\Gamma_k,\Gamma_{k+1},\cdots, \Gamma_{N}), \forall k. \nonumber
\end{align}
Here, $\phi$ is a constant weight coefficient.

\subsection{Problem $\mathcal{P}_0$ Remodeling}
First, we consider a finite-horizon discounted Markov decision process (MDP) to directly solve problem $\mathcal{P}_0$ in \eqref{problem_3}, which is defined
as $\mathcal{M} = (\mathcal{S},\mathcal{A},P,N,r,\gamma)$ below:

\begin{itemize}
	\item [1)] System States Space $\mathcal{S} = \{x^a_k\}$: The system state under the optimal sequential attack at time $k$ is designed as $x_k^a=[x_k^{\{a,1\}},x_k^{\{a,2\}},\ldots,x_k^{\{a,n\}}]^{\mathrm{T}}$, where $x_k^{\{a,i\}}\in \mathbb{R}$ is the state of the $i$-th agent under attacks.
	\item [2)] Action Space $\mathcal{A}=\{\Gamma,\theta\}$: Action space consists of the attack selection decision $\Gamma$ and the attack signal $\theta$.
	      A binary attack selection vector $\Gamma_k=[\gamma_k^1,\gamma_k^2,\ldots,\gamma_k^n]^{\mathrm{T}}$ is used to denote the attack selection decision at time $k$ with $\gamma_k^i\in \{0,1\}$. If $\gamma_k^i=1$, then the $i$-th agent is selected to be compromised at time $k$, otherwise $\gamma_k^i=0$.
	      The attack signal $\theta_k \in\mathbb{R}$ is applied for compromising states at time $k$.
	\item [3)] System Dynamics $P$: With the given state $x^a_k$, the attack selection decision $\Gamma_k$ and the attack signal $\theta_k$, the dynamics of the state are depicted as \eqref{eq2}, namely
	      \begin{equation*}
		      x_{k+1}^{a}=W_k x^{a}_k + \Gamma_k \theta_k.
	      \end{equation*}
	\item [4)] Reward function $r:\mathcal{S}\times \mathcal{A}\rightarrow \mathbb{R}$: 
	We adopt the objective function in \eqref{problem_3} to design the one-step reward function
	\begin{align}\label{reward}
		r_k(x^a_k,\Gamma_k,\theta_k|x^*, \theta^*_k)= & -\|x^a_{k+1}-x^*\|^2_{P_k} - \|\Gamma_k \theta_k\|^2_{Q_k} \nonumber \\
										& - \phi(\theta_k-\theta_k^*)^2.
	\end{align}
	\item [5)] Constants: $N\in\mathbb{R}$ is the horizon (episode length) and $\gamma\in (0,1)$ is the discount factor. 
\end{itemize}

Our goal is to find a stochastic policy $\pi:\mathcal{S}\rightarrow\mathcal{A}$.
With the given state $x^a_k$, the policy decides what action $\Gamma_k, \theta_k$ to take. 
Concretely, the objective of the adversary is to find a good policy $\pi$ to maximize the expected discounted total reward, given by
\begin{equation}
	\mathbb{E}_{\pi}\left[\sum_{k=0}^{N}\gamma^k r_k(x^a_k,\Gamma_k,\theta_k|x^*, \theta^*_k)\right].
\end{equation}

\begin{algorithm}[t]
	\caption{One-stage Learning-based Algorithm}
	\label{algo1}
	\KwIn {{The expected malicious state $x^*$; the initial state $x_0$; episode length $N$; stopping criterion $\delta$; data buffer $\mathcal{B}$}}
	\KwOut{A trained policy $\pi^{*}$}
	Initialize the policy $\pi$, last performance $J_l=0$, current performance $J_c=0$\\
	\While{$J_c-J_l<\delta$}{
		$J_l = J_c$;\\
		Initialize the $x_0$, buffer $\mathcal{B}=\emptyset$;\\
		// data collection\\
		\For {$k=0,1,2,\cdots,N$}
		{
			$\Gamma_k, \theta_k = \pi(x^a_k)$; // Decision procedure\\
			$x^{a}_{k+1} = W_k x^{a}_k + \Gamma_k \theta_k$; // Model dynamic \eqref{eq2} \\
			$\mathcal{B} \leftarrow \mathcal{B} \cup \{x^a_k,\Gamma_k,\theta_k\}$; // Store data\\
		}
		// computing rewards \\
		\For {$k=0,1,2,\cdots,N$}
		{
			Compute $\theta^*_k$ as described in Fig. \ref{fig:flow}.\\
			// adopting \eqref{eq:f} to compute $\theta^*_k$\\
			$r_k = R(x^a_k,\Gamma_k,\theta_k|x^*, \theta^*_k)$; // Reward function \eqref{reward} \\
			Substitute $(x^a_k,\Gamma_k,\theta_k)$ with $(x^a_k,\Gamma_k,\theta_k,r_k)$ in $\mathcal{B}$.
		}
		Compute the current performance $J_c=J$\\
		// model training\\
		Train the policy $\pi$ with data in $\mathcal{B}$ using PPO; See details in \cite{schulman2017proximal}.\\
	}
	\Return{$\pi^*=\pi$}
\end{algorithm}
\begin{algorithm}[t]
	\caption{Heuristic Two-stage Learning-based Joint Attack Algorithm (HTL-JA)}
	\label{algo2}
	\KwIn {{The expected malicious state $x^*$; the initial state $x_0$; the number of trajectory $T_r$.}}
	\KwOut{Sequential attack selection strategy $\Gamma$ and sequential attack signal $\theta$.}
	\textbf{Stage $1$:} Attain a sub-optimal policy $\pi$ under Algorithm \ref{algo1} with a loose stopping criterion;\\
	\textbf{Stage $2$:} 
	Generate $T_r$ alternative attack strategies $\{(\Gamma^{(1)}, \theta^{(1)}), \ldots, (\Gamma^{(T_r)}, \theta^{(T_r)})\}$ based on $\pi$;\\
	Initialize best strategy $BS=\emptyset$; current min objective $J^*=\infty$; \\
	\For {$i=1,2,\cdots,T_r$}
	{
		Compute the optimal sequential attack signal $\theta^{*(i)}$ given $\Gamma^{(i)}$ based on Theorem \ref{th:optimal};\\
		\If{$J(\Gamma^{(i)}, \theta^{*(i)})<J^*$}
		{
			$BS \leftarrow (\Gamma^{(i)}, \theta^{*(i)})$;\\
			$J^* \leftarrow J(\Gamma^{(i)}, \theta^{*(i)})$;\\
		}
	}
	\Return $BS$;
\end{algorithm}

\subsection{Learning-based Joint Attack Strategy Design}
Policy-based algorithms have achieved great success in solving complicated dynamic problems~\cite{sutton2018reinforcement, lillicrap2015continuous}, enjoying good sample efficiency.
The key idea of the policy-based algorithms is to increase the probability of the actions leading to higher rewards.
We propose a learning-based algorithm (Algorithm \ref{algo1}) that utilizes our prior knowledge about the relationship between the attack selection strategy and the attack signal in \eqref{eq:f}.
We initialize a random policy $\pi_0$ and follow the steps below to update the policy $\pi_t$ iteratively.
Our algorithm cycles through three phases, described below.
\begin{enumerate}
	\item Data collection (Line 4-10): 
	Initialize the environment state as $x^a_0$ and the buffer $\mathcal{B}=\emptyset$.
	The agent interacts with the environment by taking actions $\{\Gamma_k, \theta_k\}$ sampling from the current policy $\pi_t(x^a_k)$.
	The environment then updates the state according to the dynamic function to attain $x^a_{k+1} = P(x^a_{k}, \Gamma_k, \theta_k)$. 
	Store the data in the buffer $\mathcal{B}$.
	It moves to the next phase when an episode ends (when $k>N$).
	\item Reward computation (Line 11-17):
	The data in the data buffer $\mathcal{B}$ do not have the reward information because computing the reward in \eqref{eq:f} requires the optimal sequential attack signal $\theta^*$ which needs complete knowledge of the sequential attack selection strategy $\Gamma$.
	Thus, the optimal sequential attack signal can only be attained after a complete episode.
	We first compute the $\theta_0^*, \cdots, \theta_{N-1}^*, \theta_N^*$ with our methods described in Fig. \ref{fig:flow}.
	Then, we can compute the reward $r_k$ with \eqref{reward} using $(x^a_k,\Gamma_k,\theta_k, \theta^*_k)$.
	At last, we store the reward into the buffer $\mathcal{B}$ by substituting $(x^a_k,\Gamma_k,\theta_k)$ with $(x^a_k,\Gamma_k,\theta_k,r_k)$.
	\item Policy training (Line 19-20):
	Now we have a buffer $\mathcal{B} = \{(x^a_k,\Gamma_k,\theta_k,r_k)\}$, which contains the state-action-reward tuple.
	With this buffer, we update the policy parameters adopting the Proximal Policy Optimization (PPO) algorithm, described in \cite{schulman2017proximal}.
\end{enumerate}
We repeat these three phases until the policy is converged where the parameter $\delta$ is set as small as possible.
The convergence of the obtained solution of PPO is discussed in \cite{schulman2015trust}.
When we assume $n=3$ and $N=50$,
directly applying Algorithm \ref{algo1} to solve problem $\mathcal{P}_0$ in \eqref{problem_3} would reach the optimal performance of the objective function value $J=30$ with $4.1\times 10^{4}$ samples, which is much less than the brute-force approach with $2^{3^{50}}\approx 1.4\times 10^{45}$ samples, as demonstrated in Table \ref{table4:comparison} of Section \ref{VI-B}.

\subsection{A Two-stage Mechanism to Speed Up}
\begin{figure}[t]
	\centering
\includegraphics[width=0.50\textwidth]{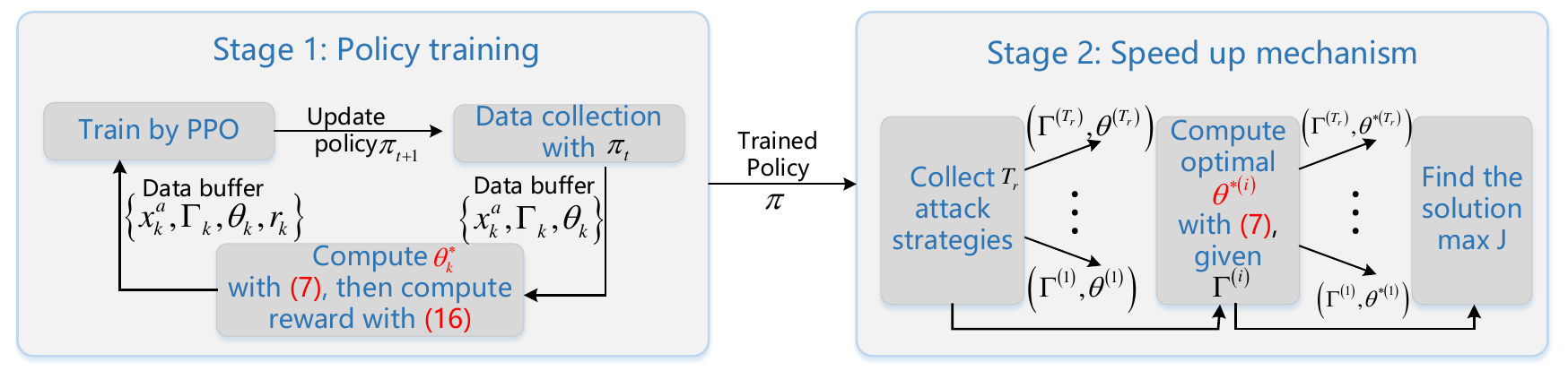}
	\caption{The schematic of Algorithm $2$ (HTL-JA)}
	\label{fig:algorithm1}
	\vspace*{-10pt}
\end{figure}

\begin{figure}[t]
	\centering	\includegraphics[width=0.35\textwidth,height=0.2\textwidth]{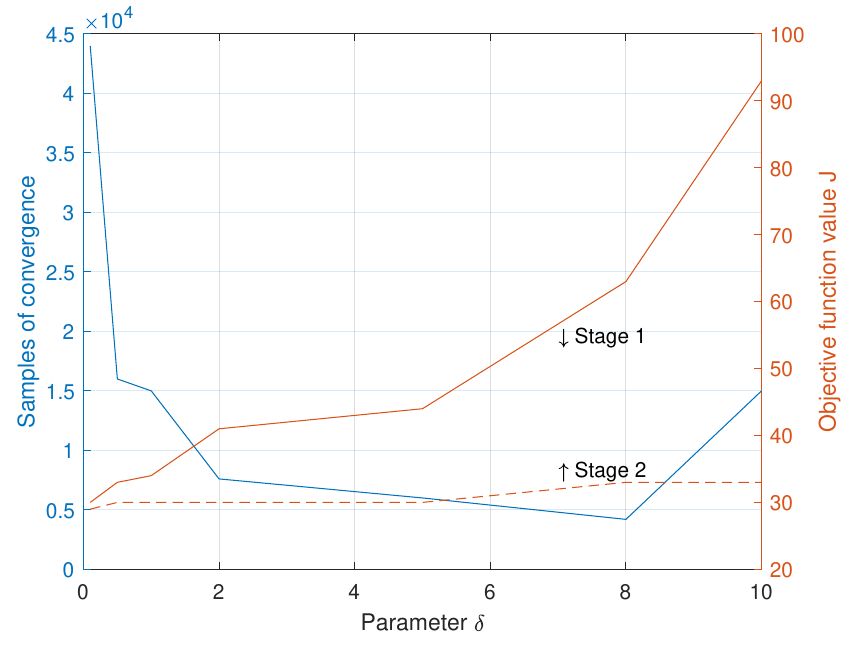}
	\caption{The effects of $\delta$ on samples and objective function value }
        \label{fig4:delta}
	\vspace*{-10pt}
\end{figure}
\vspace{-0.1cm}
Algorithm \ref{algo1} is referred to one-stage learning-based algorithm. Note that the convergence procedure of Algorithm \ref{algo1} suffers from a long tail effect: the objective function reaches the sub-optimal very quickly, but it takes multiple times of epochs to reach the optimal value. 
Thus, we utilize our system knowledge to propose a two-stage learning-based algorithm to tackle this problem and find an approximate solution with much less consumption, as shown in Algorithm \ref{algo2}.
The stage $1$ depends on Algorithm \ref{algo1} and is almost exactly the same as Algorithm \ref{algo1} except the setting of the parameter $\delta$. In stage $1$, $\delta$ is set to be larger than that in Algorithm \ref{algo1} since we aim to achieve a sub-optimal policy $\pi$ rather than attain the optimal converged policy in Algorithm \ref{algo1}. 
In stage $2$, with the sub-optimal policy $\pi$, the agent generates several alternative sequential attack selection sequences $\{\Gamma^{(1)}, \ldots, \Gamma^{(T_r)}\}$ where $T_r$ is the number of the sequential attack selection strategies.
Then we traverse all the alternative sequential attack selection strategies to compute corresponding theoretically optimal sequential attack signals $\{\theta^{*(1)}, ..., \theta^{*(T_r)}\}$ based on Theorem \ref{th:optimal}. 
The following lemma is provided to show the feasibility of the obtained solution in step $6$-$10$ of Algorithm \ref{algo2}.
\begin{lemma}
    With the optimal sequential attack signal $\theta^{*(i)}$ in Theorem \ref{th:optimal}, we have 
    \begin{align}\label{eq:36}
        J(\Gamma^{(i)},\theta^{*(i)})\leq J(\Gamma^{(i)},\theta^{(i)}),
    \end{align}
    for any $i\in \{1,2,\ldots,T_r\}$.
\end{lemma}
\begin{proof}
    Since $\theta^{*(i)}$ is the optimal solution of problem $\mathcal{P}_1$, \eqref{eq:36} can be directly validated under the given attack selection strategy $\Gamma^{(i)}$.
\end{proof}




In a word, the two-stage learning-based attack framework is provided in Fig. \ref{fig:algorithm1}, where the trained sub-optimal policy $\pi^*$ in stage $1$ paves way to speed up the procedure of reaching the optimal solution of problem $\mathcal{P}_0$ in stage $2$.
{\color{black}
\begin{remark}[Parameter $\delta$]
   To deepen our understanding of $\delta$, herein we discuss the effects of the parameter $\delta$ on the proposed one-stage and two-stage learning-based algorithms. Note that the difference between Algorithm \ref{algo1} and stage $1$ in Algorithm \ref{algo2} lies in the parameter $\delta$. When $\delta$ approaches zero, stage $1$ in Algorithm \ref{algo2} is the same as Algorithm \ref{algo1}. Consider a consensus process with three agents, shown in Section \ref{IV}. The relationship among $\delta$, the samples of convergence and the objective function value $J$ is depicted in Fig. \ref{fig4:delta}. When $\delta$ approaches zero, we find that the samples of convergence increases rapidly. In other words, Algorithm \ref{algo1} needs large number of samples to obtain the feasible solution than Algorithm \ref{algo2}. When $\delta=8$, the proposed algorithms do not reach the minimum objective function value while keeping the minimum samples. In addition, it illustrates that the two-stage learning mechanism plays an important role in accelerating the convergence process and the objective function value $J$ will converge to $30$ and decreases as $\delta$ approaches zero.   
\end{remark}

\begin{remark}[Optimality]
    Note that the policy gradient optimization method is the basis of the proposed heuristic two-stage learning-based algorithm. Hence, the optimality of the proposed algorithm also depends on that of the policy gradient method, whose convergence and optimality are analyzed in \cite[Theorem 5]{agarwal2021theory}. We find that the lower bound of the optimality error between the objective function value under the optimal solution and that under the actual solution is less than required bound $\varepsilon$. Concretely, the bound $\varepsilon$ is proportional to the dimensions of states and the discount factor $\gamma$, and inversely proportional to iterations $T$. In addition, in \cite{wang2019neural}, the upper bound of the global optimality of stationary point is demonstrated. Both works provide the feasible analysis for the convergence and optimality of the policy gradient method, thus showing the potential insights to guarantee the convergence and optimality of the proposed algorithm, which deserves further investigation in the future.  
   
   \end{remark} }

\section{Simulation Results}\label{IV}
In this section, we evaluate the performance of the optimal sequential attack signal and attack selection strategy, respectively.

\subsection{Performance of Optimal Sequential Attack Signal}\label{IV-A}
In this part, we analyze the driving performance of the obtained sequential attack signal and the inverse convergence of its critical parameters $K_k$ and $F_k$.

Consider a consensus process with three agents where the dynamics of the whole system satisfy \eqref{eq1}. We set the matrix $W=I_3-0.2*L$, which can achieve the average consensus without attacks. Meanwhile, the system is stable and controllable. In the linear network, the Laplacian matrix $L=[1~-1~0;-1~2~-1;0~-1~1]$ and $L=[2~ -1~ -1;-1~ 2~ -1;-1~ -1~ 2]$ in the circle network. Let time $N=50$, and weight matrices $P_k=Q_k=H=I_3$ for all $0\leq k\leq N$. Then we set the critical parameter $K_{N+1}=H$, the initial state $x_0=[-1~ 12~-5]^{\mathrm{T}}$ and the expected malicious state $x^{*}=[0~0~0]^{\mathrm{T}}$.


\subsubsection{Effects of the sequential attack signal $\theta$ on the system states}
Given the attack selection strategy $\Gamma_1=[1~0~0]^{\mathrm{T}}$ and the linear network, the differences between the states without attacks and that with the injected attack signal $\theta$ are shown in Fig. \ref{fig:states}. It is illustrated that the injected sequential attack signal can steer the average consensus value $[2~2~2]^{\mathrm{T}}$ to the desired malicious state $x^{*}=[0~0~0]^{\mathrm{T}}$.

\begin{figure*}[t]
	\centering
	\subfigure[The variations of states with/without attacks ]{\label{fig:states}
		\includegraphics[width=0.31\textwidth,height=0.25\textwidth]{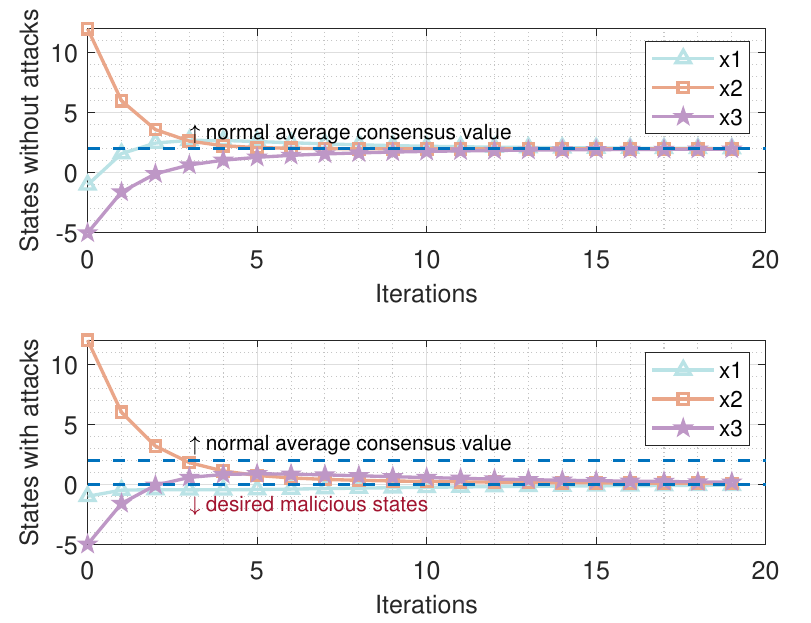}
	}
	\subfigure[Effects of $\Gamma$ on $\theta$ under different networks]{\label{fig:theta}
		\includegraphics[width=0.31\textwidth]{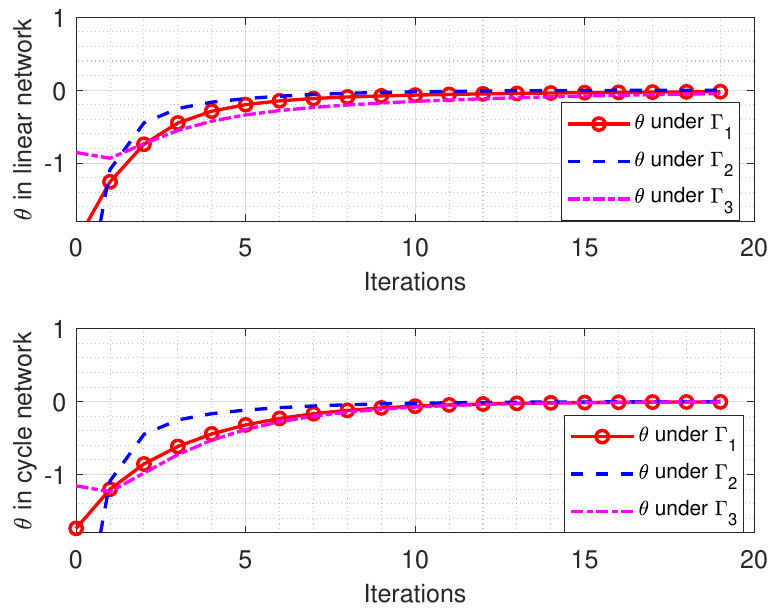}	
	}
        \subfigure[Effects of $x_0$ on $\theta$ under the cycle network]{\label{fig:initial_states}
		\includegraphics[width=0.31\textwidth]{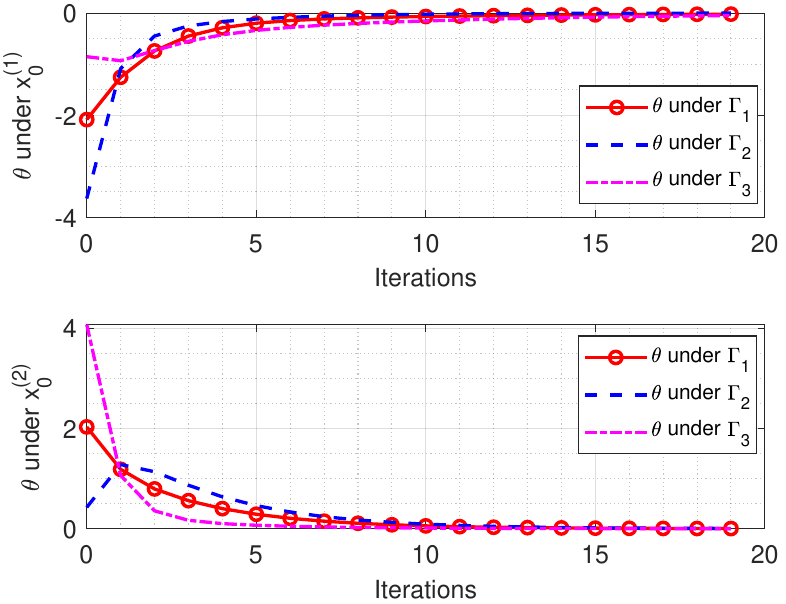}	
	}
	\caption{Performance of the optimal sequential attack signal.}
	
	\label{fig:effects}
	\vspace*{-5pt}
\end{figure*}

	
\subsubsection{Effects of the attack selection strategy $\Gamma$ on $\theta$ under different networks}
We set attack selection strategy $\Gamma_1=[1~0~0]^{\mathrm{T}}$, $\Gamma_2=[0~1~0]^{\mathrm{T}}$, and  $\Gamma_3=[0~0~1]^{\mathrm{T}}$. The effects of different attack selection strategies on the injected sequential attack signal under linear and cycle networks are shown in Fig. \ref{fig:theta}. Notably, the injected attack signal $\theta$ varies with the distinct attack selection strategies and approaches zero. Moreover, from Table \ref{table1:linear} and Table \ref{table2:cycle}, we find that there exists a trade-off between the injected attack energy and the objective function value regardless of the type of connected networks. Specifically, the more the objective function needs to be minimized while driving the states to the malicious states, the more attack energy needs to be injected.
\vspace{-0.1cm}
\begin{table}[ht]
	\caption{Results of different attack selection in linear networks}
	\label{table1:linear}
	\vspace{-10pt}
	\begin{center}
		\begin{threeparttable}
			\begin{tabular}{c c c c}
				\toprule
				\tabincell{c}{\textbf{Network} \\\textbf{structure}}  &
				\tabincell{c}{Attack selection\\
					strategy $\Gamma$}  & 
				\tabincell{c}{Attack energy \\
					$\sum_{k=0}^{N} \|\Gamma_k \theta_k \|^2$} &
				\tabincell{c}{\textbf{Objective}\\ $J$}  \\

                     \midrule
				\multirow{4}{*}{\textbf{Linear}}  
                    & $[1~ 0~ 0]^{\mathrm{T}}$ & $6.8787$ & $68.6639$\\ \cmidrule(l){2-4}

				\specialrule{0.00em}{1pt}{1pt}
				& $[0~ 1~ 0]^{\mathrm{T}}$ & $14.7073$ & $36.0239$\\ \cmidrule(l){2-4}
				
				\specialrule{0.00em}{1pt}{1pt}
				& $[0~ 0~ 1]^{\mathrm{T}}$ & $3.0517$ & $130.6101$\\
				
                    \bottomrule
			\end{tabular}
		\end{threeparttable}
	\end{center}
	\vspace{-10pt}
\end{table}

\begin{table}[ht]
	\caption{Results of different attack selection in circle networks}
	\label{table2:cycle}
	\vspace{-10pt}
	\begin{center}
		\begin{threeparttable}
			\begin{tabular}{c c c c}
				\toprule
				\tabincell{c}{\textbf{Network}\\ \textbf{structure}}  &
				\tabincell{c}{Attack selection\\
					strategy $\Gamma$}  & 
				\tabincell{c}{Attack energy \\
					$\sum_{k=0}^{N} \|\Gamma_k \theta_k \|^2$} &
				\tabincell{c}{\textbf{Objective}\\ $J$}  \\

                    \midrule
				\multirow{4}{*}{\textbf{Circle}}  & $[1~ 0~ 0]^{\mathrm{T}}$ & $5.9828$ & $64.0186$\\ \cmidrule(l){2-4}
				
				\specialrule{0.00em}{1pt}{1pt}
				& $[0~ 1~ 0]^{\mathrm{T}}$ & $14.7073$ & $23.3255$\\ \cmidrule(l){2-4}
				
				\specialrule{0.00em}{1pt}{1pt}
				& $[0~ 0~ 1]^{\mathrm{T}}$ & $4.9348$ & $72.1756$\\
				
                    \bottomrule
			\end{tabular}
		\end{threeparttable}
	\end{center}
	\vspace{-10pt}
\end{table}
\vspace{0.1cm}

\subsubsection{Effects of the initial states on $\theta$}
We set two types of initial states $x_0^{(1)}=[-1~ 12~-5]^{\mathrm{T}}$ and $x_0^{(2)}=[-1~ 10~-15]^{\mathrm{T}}$, and remain the other conditions. The effects of the initial states on the injected optimal sequential attack signal $\theta$ are shown in Fig. \ref{fig:initial_states}. It is illustrated that the size of the injected attack signal highly depends on the initial states.
Even though there exists the same initial state for agent $1$, the size of the injected attack signal is different and influenced by the initial states of other agents.


\subsubsection{Inverse convergence of $K_k$ and $F_k$}
In this part, we show the inverse convergence of $K_k$ and $F_k$, which are measured by the following index $K_{c}\triangleq \|K_k-K^{\star}\|$ and $F_{c}\triangleq \|F_k-F^{\star}\|$,
where $K^{\star}$ and $F^{\star}$ are the steady-state matrix of $K_k$ and $F_k$ for $0\leq k\leq N$, respectively. Given the attack selection strategy $\Gamma_1=[1~ 0~ 0]^{\mathrm{T}}$ and the other same conditions as the first part, the convergence error of $K_k$ and $F_k$ are illustrated as Fig. \ref{fig:convergence_linear} and Fig. \ref{fig:convergence_cycle}. Under the linear network, when the first or the third agent is compromised, the convergence error of $K_k$ and $F_k$ are the same, which is different from that when the only second agent is attacked. In other words, the effects of attack selection strategies on the injected attack signal depend on the network structure.
Especially, under the cycle network, the selection of the compromised agents does not affect the injected signal. Moreover, comparing Fig. \ref{fig:convergence_linear} with Fig. \ref{fig:convergence_cycle},
it is easy to reveal that the convergence rate of $F_k$ is greater than that of $K_k$, which is owing to the convergence of weight matrix $W_k$. From Table \ref{table3:time}, we show the inverse convergence times for $K_k$ and $F_k$, which validate the result in Corollary \ref{c2}. Meanwhile, we find that only $15$ iteration times are required to compute $K_k$ and $14$ iteration times for $F_k$ regardless of the length of $N$.
\vspace{-0.2cm}
\begin{table}[ht]
	\caption{Times of inverse convergence of $K_k$ and $F_k$}
	\label{table3:time}
	\vspace{-0.5cm}
	\begin{center}
		\begin{threeparttable}
			\begin{tabular}{c c c c c}
				\toprule
				\textbf{Length of $N$} &
				$50$                   & $100$    & $200$     & $1000$    \\
                    \midrule
				\tabincell{c}{\textbf{Inverse Convergence}                \\ \textbf{Time of $K_k$}}  &
				$[1,35]$               & $[1,85]$ & $[1,185]$ & $[1,985]$ \\

				\specialrule{0.00em}{3pt}{1pt}
				\tabincell{c}{\textbf{Inverse Convergence}                \\ \textbf{Time of $F_k$}}  &
				$[1,36]$               & $[1,86]$ & $[1,186]$ & $[1,186]$ \\
                    \bottomrule
			\end{tabular}
		\end{threeparttable}
	\end{center}
	\vspace{-10pt}
\end{table}
\vspace{-0.5cm}

\begin{figure}[t]
	\centering
 \vspace{-0.2cm}
	\subfigure[Linear network ]{\label{fig:convergence_linear}
		\includegraphics[width=0.22\textwidth]{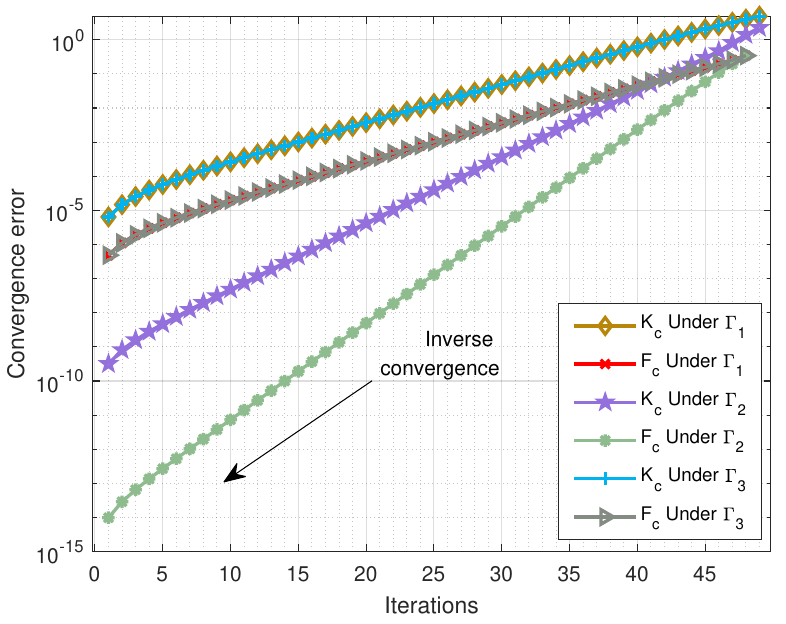}
	}
	\subfigure[Cycle network]{\label{fig:convergence_cycle}
		\includegraphics[width=0.22\textwidth,height=0.174\textwidth]{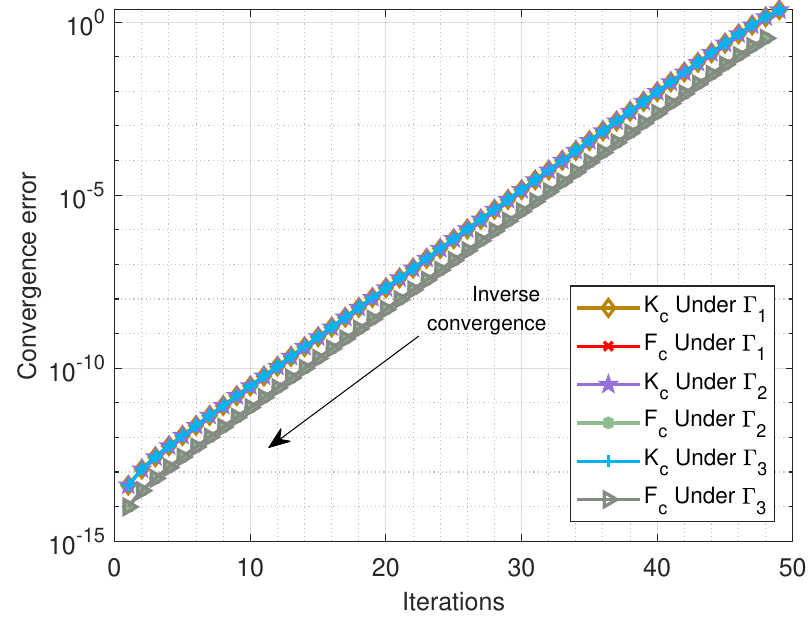}	
	}
	\caption{The convergence error of $K_k$ and $F_k$ under different networks.}
	
	\label{fig:convergence}
	\vspace*{-5pt}
\end{figure}

\subsection{Performance of Learning-based Attack Selection Strategy}\label{VI-B}
\subsubsection{Parameter setting}
Consider the aforementioned consensus process with three agents and the same parameter settings.
Meanwhile, the stopping criterion $\delta$ in Algorithm $1$ and the number of trajectory $T_r$ in Algorithm $2$ are set as $\delta=2$ and $T_r=10$, respectively. In addition, we consider the consensus process with ten agents where the Laplacian matrix is designed as $L =[2~ -1~ 0~ 0~ 0~ 0~ 0~ 0~ 0~ -1;
	-1~ 3~ -1~ -1~ 0~ 0~ 0~ 0~ 0~ 0;
	0~ -1~ 2~ 0~ 0~ 0~ -1~ 0~ 0~ 0;
	0~ -1~ 0~ 3~ -1~ -1~ 0~ 0~ 0~ 0;
	0~ 0~ 0~ -1~ 1~ 0~ 0~ 0~ 0~ 0;
	0~ 0~ 0~ -1~ 0~ 2~ 0~ 0~ -1~ 0;
	0~ 0~ -1~ 0~ 0~ 0~ 2~ -1~ 0~ 0;
	0~ 0~ 0~ 0~ 0~ 0~ -1~ 2~ -1~ 0;
	0~ 0~ 0~ 0~ 0~ -1~ 0~ -1~ 2~ 0;
	-1~ 0~ 0~ 0~ 0~ 0~ 0~ 0~ 0~ 1]$,
and the initial state is set as $x_0=[-1, 12, -5, 5, 2, 7, 7, 0, 9,-10]^{\mathrm{T}}$. We remain the other conditions and set $\delta=0.1$ for Algorithm \ref{algo1} and $\delta=2$ for Algorithm \ref{algo2}.
\subsubsection{Compared algorithms}
In this part, we compare five kinds of algorithms. 
The first approach is the brute force method.
The second approach is the random selection strategy, where both the sequential attack selection strategy $\Gamma$ and the attack signal $\theta$ are randomly generated. The third approach is the sampling-based algorithm where $\Gamma$ is randomly sampled and exploited to compute $\theta$ based on Theorem \ref{th:optimal}, generating the optimal strategy after multiple samples.
The fourth approach is to apply Algorithm \ref{algo1} to obtain the solution of the attack selection strategy and the injected attack signal, respectively. 
The fifth approach is our algorithm, i.e., Algorithm \ref{algo2}, where the prior information about the attack signal in \eqref{theta_without} is used to evaluate and the two-stage learning-based mechanism is adopted to speed up the convergence process.
\subsubsection{Results}
As depicted in Fig. \ref{fig:simulation_3}, we need more than $4.1 \times 10^4$ samples if Algorithm \ref{algo1} is desired to converge to the optimal solution ($J\approx 30$). 
However, if we adopt the two-stage mechanism, we can jump out the RL iteration at a sub-optimal solution ($J \approx 40$) and refine it with our stage $2$ in Algorithm \ref{algo2} to reach $J=30$.
This means that we can reach the optimal solution with only $8.0\times 10^3$ samples, much less than the one-stage learning-based algorithm (Algorithm \ref{algo1}). The result of comparison among several algorithms is shown in Table \ref{table4:comparison}. Note that the sampling-based selection strategy seems better than the proposed algorithm under the three agents. Probably because the number of agents is too small, the considered scenario is simple and easy to obtain the optimal solution based on the random sampling. When ten agents are considered, Table \ref{table4:comparison} further validates the conjecture since the sampling-based strategy does not work well. We find that the effectiveness of the proposed algorithm is not affected by the number of agents. Algorithm \ref{algo2} can reach the optimal objective function value with the minimum samples compared to Algorithm \ref{algo1}. Meanwhile, compared with the sampling-based strategy, Algorithm \ref{algo2} obtains lower objective function value while the time complexity is low. 
\begin{figure}
	\centering
 \vspace{-0.6cm}
	\includegraphics[width=0.40\textwidth]{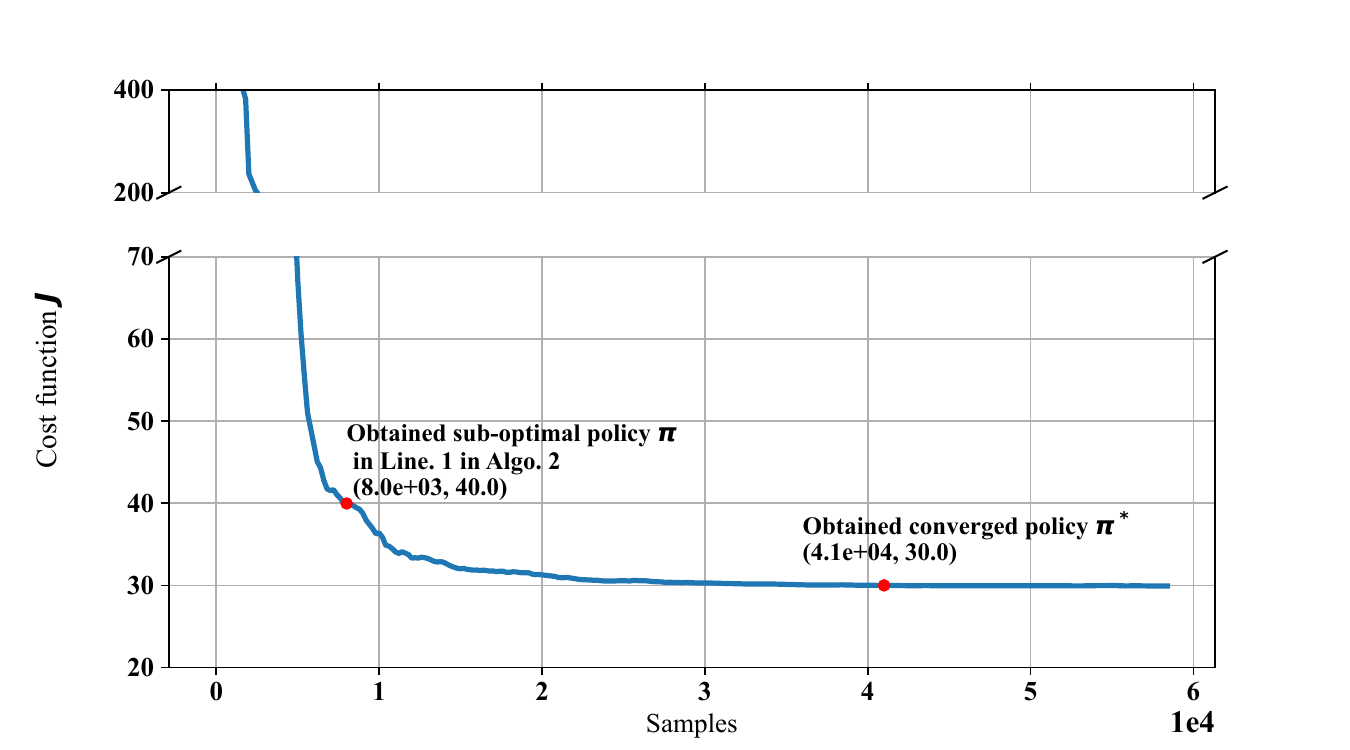}
	\caption{The result of learning-based algorithms}
	\label{fig:simulation_3}
\end{figure}

\vspace{-0.2cm}
 \begin{table}[ht]
	\caption{Comparisons of Algorithms under Different Amount of Agents}
	\label{table4:comparison}
	\vspace{-12pt}
	\begin{center}
		\begin{threeparttable}
			\begin{tabular}{l l l l l}
				\toprule
    \textbf{$\#$Agents} & \multicolumn{2}{c}{THREE} & \multicolumn{2}{c}{TEN}\\
                    \midrule
				\multicolumn{1}{l|}{\textbf{Algorithms}} & \tabincell{c}{Time Complexity\\ (Sample Size) $\downarrow$} & \multicolumn{1}{l|}{$J$ $\downarrow$} &\tabincell{c}{Time\\ Complexity $\downarrow$} & $J$ $\downarrow$\\
                    \midrule
				\tabincell{c}{\textbf{Brute force}}  & $1.4\times 10^{45}$ & N/A & $1.0\times 10^{53}$ & N/A\\ 
    
                    \specialrule{0.00em}{3pt}{1pt}
                    \textbf{Random} & $1.0\times 10^{6}$ & $4039$ & $1.0\times 10^{6}$ & $21311$\\

                    \specialrule{0.00em}{3pt}{1pt}
                    \tabincell{c}{\textbf{Sampling}\\
                    \textbf{-based}}  & $1.0\times 10^{5}$ & $\textbf{27.8}$ & $1.0\times 10^{5}$ & $318.0$\\
                    
				\specialrule{0.00em}{3pt}{1pt}
				\textbf{Algorithm \ref{algo1}} & $4.1\times 10^{4}$ & $30.0$ & $5.1\times 10^{4}$ & $\textbf{295.0}$\\ 
                    \specialrule{0.00em}{3pt}{1pt}
                    \textbf{Algorithm \ref{algo2}} & $\pmb{8.0\times 10^3}$ & $30.0$ & $\pmb{1.8\times 10^4}$ & $\textbf{295.0}$\\ 
                    \bottomrule
			\end{tabular}
		\end{threeparttable}
	\end{center}
	\vspace{-10pt}
\end{table}
\vspace{-0.2cm}



\section{Conclusion}\label{V}
We constructed an optimal sequential false data injection attack co-design framework where the injected attack signals and attack selection strategy are strongly coupled optimization variables and vary with the sampling time in discrete-time systems. Specifically, we first derived an optimal sequential attack signal, which showcases the closed-form explicit expression between the injected attack signal and the attack selection strategy. In addition, we proved the inverse convergence of the critical parameters in the optimal sequential attack signal. Furthermore, with the prior knowledge of the closed-form relationship, we proposed the heuristic learning-based attack algorithms to obtain the sequential feasible solution.
Future work will strive to design the resilient algorithms to defend the system against the proposed optimal sequential FDI attacks.

\appendix
\subsection{Proof of Theorem \ref{th:optimal}}\label{APP-A}

The proof can be completed by solving the Bellman equation backward from time $N+1$ of termination.

When time $k=N+1$, $K_{N+1}=H$, for any $x^{a}_{N+1}\in \mathbb{R}^{n}$, the value function
\vspace{-0.2cm}
\begin{align}\label{N+1}
	   ~~~V(x^{a}_{N+1}&,N+1)\nonumber                               \\
	= & (x^{a}_{N+1}-x^{*})^{\mathrm{T}}H(x^{a}_{N+1}-x^{*})\nonumber \\
	= & (x^{a}_{N+1})^{\mathrm{T}}K_{N+1}(x^{a}_{N+1}) + G_{N+1},
\end{align}
where $G_{N+1}=-2(x^{a}_{N+1})^{\mathrm{T}}K_{N+1}x^{*}+\|x^{*}\|^2$.
Note that the value function $V(x^{a}_{N+1},N+1)$ is the quadratic function with respect to $x^{a}_{N+1}$. Next, with the mathematical induction method, we prove that the value function always satisfies the following form
\vspace{-0.2cm}
\begin{align}\label{ite}
	V(x^{a}_{k+1},k+1)=(x^{a}_{k+1})^{\mathrm{T}}K_{k+1}(x^{a}_{k+1}) + G_{k+1},
\end{align}
where $K_{k+1}$ is the real symmetric positive definite matrix for $k=0,1,\ldots,N$.

Then, we derive the optimal attack signal $\theta_N$ at time $N$. With the obtained value function $ V(x^{a}_{N+1},N+1)$ in \eqref{N+1}, for any $x^{a}_{N}\in \mathbb{R}^{n}$, we have
\vspace{-0.2cm}
\begin{align}\label{N}
	 & ~~~V(x^{a}_{N},N)\nonumber                                                                                         \\
	 & = \mathrm{min}_{\theta_N} \left\{(x^{a}_N-x^{*})^{\mathrm{T}}P_N(x^{a}_N-x^{*}) \right.\nonumber                   \\
	 & ~~~+\left. \|\Gamma_N \theta_N\|^2_{Q_N} +V(x^{a}_{N+1},N+1)
	\right\}\nonumber                                                                                                     \\
	 & = \mathrm{min}_{\theta_N} \left\{
	(x^{a}_N-x^{*})^{\mathrm{T}}P_N(x^{a}_N-x^{*}) \right.\nonumber                                                       \\
	 & ~~~+\left.  (W_N x^{a}_{N} + \Gamma_N \theta_N)^{\mathrm{T}} K_{N+1} (W_N x^{a}_{N} + \Gamma_N \theta_N) \nonumber \right.           \\
	 & ~~~+ \left. \theta_N^{\mathrm{T}} \Gamma_N^{\mathrm{T}} Q_N \Gamma_N \theta_N + G_{N+1}\right\}.
  \vspace{-0.2cm}
\end{align}

Taking the derivative of \eqref{N} with respect to $\theta_N$, for any $x^{a}_{N}\in \mathbb{R}^{n}$, we have
$2 \theta_N^{\mathrm{T}} \Gamma_N^{\mathrm{T}} Q_N \Gamma_N + 2(W_N x^{a}_{N} + \Gamma_N \theta_N)^{\mathrm{T}} K_{N+1} \Gamma_N = 0$. Thus, it can be inferred that
\vspace{-0.2cm}
\begin{align}\label{theta-N}
	\theta_N = -R_N^{-1}(\Gamma_N^{\mathrm{T}} K_{N+1} W_N x^{a}_{N} - \Gamma_N^{\mathrm{T}} K_{N+1} x^{*}),
\end{align}
where $R_N \triangleq \Gamma_N^{\mathrm{T}}(Q_N+K_{N+1})\Gamma_N$.
\eqref{theta-N} is rewritten as
\vspace{-0.2cm}
\begin{align}\label{theta-NN}
	\theta_N = F_N x^{a}_N + M_N,
\end{align}
where $F_N= -[\Gamma_N^{\mathrm{T}}(Q_N+K_{N+1})\Gamma_N]^{-1} \Gamma_N^{\mathrm{T}} K_{N+1} W_N$ and $M_N= [\Gamma_N^{\mathrm{T}}(Q_N+K_{N+1})\Gamma_N]^{-1} \Gamma_N^{\mathrm{T}} K_{N+1} x^{*}$.

When time $k=N$, combined with \eqref{N} and \eqref{theta-NN}, we derive the value function
\vspace{-0.2cm}
\begin{align}\label{N1}
	 & ~~~V(x^{a}_{N},N)\nonumber                                                                   \\
	 & = (x^{a}_N)^{\mathrm{T}} \left\{
	P_k + W_k^{\mathrm{T}} K_{K+1} W_k + 2 W_k^{\mathrm{T}} K_{k+1} \Gamma_k F_k \right. \nonumber  \\
	 & ~~~+ \left. F_k^{\mathrm{T}} \Gamma_k^{\mathrm{T}} (Q_k + K_{k+1} ) \Gamma_k F_k
	\right\}(x^{a}_N)+ G_{N+1}  \nonumber                                                           \\
	 & ~~~- 2 (x^{*})^{\mathrm{T}} P_N x^{a}_N + \|x^{*}\|^2 \nonumber                              \\
	 & ~~~+ \theta_N^{\mathrm{T}} \Gamma_N^{\mathrm{T}} (Q_N + K_{N+1}) \Gamma_N \theta_N \nonumber \\
	 & ~~~+ 2 x_N^{\mathrm{T}} W_N^{\mathrm{T}} K_{N+1} \Gamma_N M_N.
\end{align}
Let
\vspace{-0.2cm}
\begin{align*}
	K_N \triangleq & P_N + W_N^{\mathrm{T}} K_{N+1} W_N
	+ 2 W_N^{\mathrm{T}} K_{N+1} \Gamma_N F_N                                               \\
	               & + F_N^{\mathrm{T}} \Gamma_N^{\mathrm{T}} (Q_N + K_{N+1} ) \Gamma_N F_N
\vspace{-0.2cm}
\end{align*}
and
\vspace{-0.2cm}
\begin{align*}
	G_N \triangleq & G_{N+1} -2 (x^{*})^{\mathrm{T}} P_N x^{a}_N + 2 x_N^{\mathrm{T}} W_N^{\mathrm{T}} K_{N+1} \Gamma_N M_N \\
	               & + \theta_N^{\mathrm{T}} \Gamma_N^{\mathrm{T}} (Q_N + K_{N+1}) \Gamma_N \theta_N + \|x^{*}\|^2.
\end{align*}
Thus, the value function $V(x^{a}_{N},N)$ also satisfies \eqref{ite}.

Then, we derive the optimal attack signal $\theta_{N-1}$ at time $N-1$. With the obtained value function $ V(x^{a}_{N},N)$ in \eqref{N}, for any $x^{a}_{N-1}\in \mathbb{R}^{n}$, we have
\vspace{-0.2cm}
\begin{align}\label{N-1}
	 & ~~~V(x^{a}_{N-1},N-1) \nonumber                                                                                     \\
	 & = \mathrm{min}_{\theta_{N-1}} \left\{(x^{a}_{N-1}-x^{*})^{\mathrm{T}}P_{N-1}(x^{a}_{N-1}-x^{*}) \right. \nonumber   \\
	 & ~~~+\left. \|\Gamma_{N-1} \theta_{N-1}\|^2_{Q_{N-1}} +V(x^{a}_{N},N)
	\right\}\nonumber                                                                                                      \\
	 & = \mathrm{min}_{\theta_{N-1}} \left\{
	(x^{a}_{N-1}-x^{*})^{\mathrm{T}}P_{N-1}(x^{a}_{N-1}-x^{*}) \right. \nonumber                                           \\
	 & ~~~+ \left. \theta_{N-1}^{\mathrm{T}} \Gamma_{N-1}^{\mathrm{T}} Q_{N-1} \Gamma_{N-1} \theta_{N-1} \nonumber \right. \\
	 & ~~~+\left.  (W_{N-1} x^{a}_{N-1} + \Gamma_{N-1} \theta_{N-1})^{\mathrm{T}} \right. \nonumber                        \\
	 & ~~~\left. K_{N} (W_N x^{a}_{N-1} + \Gamma_{N-1} \theta_{N-1}) + G_{N}\right\}.
\end{align}
Taking the derivative of \eqref{N-1} with respect to $\theta_{N-1}$, for any $x^{a}_{N-1}\in \mathbb{R}^{n}$, we have
$2 \theta_{N-1}^{\mathrm{T}} \Gamma_{N-1}^{\mathrm{T}} Q_{N-1} \Gamma_{N-1} + 2(W_{N-1} x^{a}_{N-1} + \Gamma_{N-1} \theta_{N-1})^{\mathrm{T}} K_{N} \Gamma_{N-1} = 0$. Thus, it can be inferred that
\vspace{-0.2cm}
\begin{align}\label{theta-N-1}
	\theta_{N-1} = & -R_{N-1}^{-1}(\Gamma_{N-1}^{\mathrm{T}} K_{N} W_{N-1} x^{a}_{N-1} \nonumber        \\
	               & - \Gamma_{N-1}^{\mathrm{T}} K_{N}P_Nx^{*} +W_N^{\mathrm{T}} K_{N+1} \Gamma_N M_N),
\end{align}
which also can be derived as $\theta_{N-1} = F_{N-1} x^{a}_{N-1} + M_{N-1}$ with $F_{N-1}= -R_{N-1}^{-1} \Gamma_{N-1}^{\mathrm{T}} K_{N} W_{N-1}$ and
\vspace{-0.2cm}
\begin{align*}
	M_{N-1}= R_{N-1}^{-1} \Gamma_{N-1}^{\mathrm{T}} K_{N} (P_Nx^{*} - W_N^{\mathrm{T}} K_{N+1} \Gamma_N M_N).
\end{align*}
When time $k=N-1$, combined \eqref{N-1} with \eqref{theta-N-1}, we derive the value function
\vspace{-0.2cm}
\begin{align*}
	V(x^{a}_{N-1},N-1)=(x^{a}_{N-1})^{\mathrm{T}}K_{N-1}(x^{a}_{N-1}) + G_{N-1},
\end{align*}
where
\vspace{-0.2cm}
\begin{align*}
	K_{N-1} & = P_{N-1} + W_{N-1}^{\mathrm{T}} K_{N} W_{N-1} + 2 W_{N-1}^{\mathrm{T}} K_{N} \Gamma_{N-1} F_{N-1} \\
	        & ~~~+ F_{N-1}^{\mathrm{T}} \Gamma_{N-1}^{\mathrm{T}} (Q_{N-1} + K_{N} ) \Gamma_{N-1} F_{N-1}
\end{align*}
and
\vspace{-0.2cm}
\begin{align*}
	G_{N-1} & =G_{N} -2 (x^{*})^{\mathrm{T}} P_{N-1} x^{a}_{N-1}                                                                 \\
	        & ~~~+ \|x^{*}\|^2 + \theta_{N-1}^{\mathrm{T}} \Gamma_{N-1}^{\mathrm{T}} (Q_{N-1} + K_{N}) \Gamma_{N-1} \theta_{N-1} \\
	        & ~~~+ 2 x_{N-1}^{\mathrm{T}} W_{N-1}^{\mathrm{T}} K_{N} \Gamma_{N-1} M_{N-1}.
\end{align*}

Continue the iterative process for $k=0,1,\ldots, N-2$. Finally, we can obtain the optimal sequential attack signal
\vspace{-0.2cm}
\begin{align*}
	\theta_k = F_k x^{a}_k + M_k,
\end{align*}
and the value function
\vspace{-0.2cm}
\begin{align*}
	V(x^{a}_{k+1},k+1)=(x^{a}_{k+1})^{\mathrm{T}}K_{k+1}(x^{a}_{k+1}) + G_{k+1},
\end{align*}
Thus, the proof is completed.

\subsection{Proof of Lemma \ref{l1}}\label{B}

The proof can be divided into two parts. One is to show the Hermitian matrix $K_k$. The other is to show $K_k\succ 0$. Both are based on the mathematical induction method.

	\textbf{Hermitian}. Let $P_k=Q_k=H_k=I$ for $k=0,1,\ldots,N$. Since $W_k$ is a real symmetrical matrix, we have $W_k=W^{*}_k$.
	When $k=N$, we have $K_{N+1}=I$, which is a real symmetrical matrix.
	We assume that $K_{1}$ is a real symmetrical matrix. Then when $k=0$, it holds that
 \vspace{-0.2cm}
	\begin{align}\label{K0}
		K_0= & P_0 + W_0^{\mathrm{T}} K_{1} W_0 + 2 W_0^{\mathrm{T}} K_{1} \Gamma_0 F_0\nonumber \\
		     & + F_0^{\mathrm{T}} \Gamma_0^{\mathrm{T}} (Q_0 + K_{1} ) \Gamma_0 F_0.
	\end{align}
	Since $F_0=-R_0^{-1}\Gamma_0^{\mathrm{T}} K_{1} W_0 $ is a real matrix
	with $R_0=\Gamma_0^{\mathrm{T}}(Q_0+K_{1})\Gamma_0=R_0^*$, then it can be inferred that
 \vspace{-0.2cm}
	\begin{align*}
		K_0^*= & P_0^* + W_0^{*} K_{1} W_0 + F_0^{*} \Gamma_0^{*} (Q_0 + K_{1} ) \Gamma_0 F_0 \nonumber             \\
		       & + 2F_0^* \Gamma_0^* K_1 W_0                                                                        \\
		=      & P_0 + W_0 K_{1} W_0 + F_0^{\mathrm{T}} \Gamma_0^{\mathrm{T}} (Q_0 + K_{1} ) \Gamma_0 F_0 \nonumber \\
		       & - 2W_0^* K_1 \Gamma_0 R_0^{-1} \Gamma_0^* K_1 W_0                                                  \\
		=      & P_0 + W_0 K_{1} W_0 + F_0^{\mathrm{T}} \Gamma_0^{\mathrm{T}} (Q_0 + K_{1} ) \Gamma_0 F_0           \\
		       & - 2W_0^{\mathrm{T}} K_1 \Gamma_0 R_0^{-1} \Gamma_0^{\mathrm{T}} K_1 W_0                            \\
		=      & K_0.
	\end{align*}
	Thus, $K_0$ is also a real symmetrical matrix. In summary, $K_k$ is a Hermitian matrix for $k=0,1,\ldots,N$.

	\textbf{Positive definite}. When $k=N$, we have $K_{N+1}=I$, which is a positive definite matrix.
	We assume that $K_{1}$ is a positive definite matrix. Then we need to prove $K_0 \succ 0$, i.e.,
 \vspace{-0.2cm}
	{\small{
				\begin{align*}
					P_0 + W_0^{\mathrm{T}} K_{1} W_0 + F_0^{\mathrm{T}} \Gamma_0^{\mathrm{T}} (Q_0 + K_{1} ) \Gamma_0 F_0 + 2 W_0^{\mathrm{T}} K_{1} \Gamma_0 F_0 \succ 0.
				\end{align*}}}
	For the third and fourth parts of $K_0$, we have
 \vspace{-0.2cm}
	\begin{align}\label{F_0}
		 & ~~~~F_0^{\mathrm{T}} \Gamma_k^{\mathrm{T}} (Q_0 + K_{1} ) \Gamma_0 F_0 +  2W_0^{\mathrm{T}} K_{1} \Gamma_0 F_0 \nonumber                                         \\
		 & ~= W_0^{\mathrm{T}} K_{1} \Gamma_0 R_0^{-1} \Gamma_0^{\mathrm{T}} (Q_0 + K_{1} ) \Gamma_0 R_0^{-1} \Gamma_0^{\mathrm{T}} K_{1} W_0 \nonumber                     \\
		 & ~~~- 2W_0^{\mathrm{T}} K_{1} \Gamma_0 R_0^{-1} \Gamma_0^{\mathrm{T}} K_{1} W_0 \nonumber                                                                         \\
		 & \overset{(s.1)}{=} R_0^{-1} W_0^{\mathrm{T}} K_{1} \Gamma_0 [\Gamma_0^{\mathrm{T}} (Q_0 + K_{1} ) \Gamma_0 R_0^{-1} - 2]\Gamma_0^{\mathrm{T}} K_{1} W_0\nonumber \\
		 & \overset{(s.2)}{=} (R_0^{-1})^2 W_0^{\mathrm{T}} K_{1} \Gamma_0 [\Gamma_0^{\mathrm{T}} (Q_0 + K_{1} ) \Gamma_0 - 2R_0]\Gamma_0^{\mathrm{T}} K_{1} W_0\nonumber   \\
		 & ~= -R_0^{-1} W_0^{\mathrm{T}} K_{1} \Gamma_0 \Gamma_0^{\mathrm{T}} K_{1} W_0,
	\end{align}
	where $(s.1)$ follows that $R_0^{-1}$ is a positive real number and $(s.2)$ exploits $R_0^{-1} R_0 =1$.
	Combined with \eqref{F_0}, we next need to prove
 \vspace{-0.2cm}
	\begin{align}\label{non-negative}
		P_0 + W_0^{\mathrm{T}} K_{1} W_0 -R_0^{-1} W_0^{\mathrm{T}} K_{1} \Gamma_0 \Gamma_0^{\mathrm{T}} K_{1} W_0 \succ 0.
	\end{align}

	Note that the sum of the positive definite matrix is still a positive definite matrix and $P_0$ is the given positive definite matrix. We only need to show the positive definiteness of the second and third parts of $K_0$. Consider that $W_0^{\mathrm{T}} K_1 W_0-R_0^{-1} W_0^{\mathrm{T}} K_{1} \Gamma_0 \Gamma_0^{\mathrm{T}} K_{1} W_0$ in \eqref{non-negative} can be rewritten as
 \vspace{-0.2cm}
	\begin{align*}
		 & (W_0^{\mathrm{T}}- R_0^{-1} W_0^{\mathrm{T}} K_{1} \Gamma_0 \Gamma_0^{\mathrm{T}} )K_1 (W_0 -  R_0^{-1} \Gamma_0 \Gamma_0^{\mathrm{T}} K_{1}W_0^{\mathrm{T}} ) \\
		 & + (W_0^{\mathrm{T}}- R_0^{-1} W_0^{\mathrm{T}} K_{1} \Gamma_0 \Gamma_0^{\mathrm{T}})K_1(R_0^{-1}\Gamma_0 \Gamma_0^{\mathrm{T}} K_{1}W_0 ).
	\end{align*}
	Let $Z_0\triangleq W_0 -  R_0^{-1} \Gamma_0 \Gamma_0^{\mathrm{T}} K_{1}W_0^{\mathrm{T}}$. Then the above result can be transformed as
 \vspace{-0.2cm}
	\begin{align*}
		Z_0^{\mathrm{T}}K_1 Z_0 + R_0^{-1} W_0^{\mathrm{T}} K_1 \Gamma_0 (1- R_0^{-1}\Gamma_0^{\mathrm{T}} K_{1} \Gamma_0 ) \Gamma_0^{\mathrm{T}} K_{1}W_0.
	\end{align*}
	Since $R_0=\Gamma_0^{\mathrm{T}}(Q_0+K_{1})\Gamma_0$, we have $1-R_0^{-1}\Gamma_0^{\mathrm{T}} K_{1} \Gamma_0>0$.
	For any non-zero vector $v_0\in \mathbb{R}^{n}$, $Z_0v_0 \neq \mathbf{0}$ holds. Since $K_1$ is a positive definite matrix, there exists
 \vspace{-0.2cm}
	\begin{align*}
		(Z_0v_0)^{\mathrm{T}} K_1 (Z_0v_0) > 0.
	\end{align*}
	Thus, we have $v_0^{\mathrm{T}} (Z_0^{\mathrm{T}} K_1 Z_0) v_0>0$. Since $v_0 \neq \mathbf{0}$, it can be inferred that $Z_0^{\mathrm{T}} K_1 Z_0$ is a positive definite matrix. Similarly, $R_0^{-1} W_0^{\mathrm{T}} K_1 \Gamma_0 (1- R_0^{-1}\Gamma_0^{\mathrm{T}} K_{1} \Gamma_0 ) \Gamma_0^{\mathrm{T}} K_{1}W_0>0$ always holds. Hence, \eqref{non-negative} is proved and we have $K_k=K^{*}_{k} \succ 0$ for $k=0,1,\cdots,N$. The proof is completed.

\subsection{Proof of Lemma \ref{l3}}\label{C}
Note that \eqref{eq:23} is a sufficient condition to ensure that $\tilde{V}_{k-1}-\tilde{V}_k$ decreases along the convergence direction of matrix $K_k-K^{\star}$ in the discrete-time system. Moreover, given $\varphi_k$ in \eqref{range}, if and only if $\tilde{V}_k=0$, the equality will be zero. Then, \eqref{eq:23} can be expressed as
\vspace{-0.2cm}
	\begin{align*}
		\tilde{V}_{k-1} = \tilde{V}_{k}-\varphi_k \tilde{V}_k^{\alpha}=\tilde{V}_{k}(1-\frac{\varphi_k}{\tilde{V}_k^{1-\alpha}}).
	\end{align*}
 
	Let the initial value of the Lyapunov function be
 \vspace{-0.2cm}
	\begin{align*}   \tilde{V}_N=\beta_N(\varphi_N)^{\frac{1}{1-\alpha}}, ~\beta_N>0.
	\end{align*}
	Substituting the value $\tilde{V}_N$ in \eqref{eq:23}, one gets
 \vspace{-0.2cm}
	\begin{align*}
		\tilde{V}_{N-1}= \beta_N(\varphi_N)^{\frac{1}{1-\alpha}} -\varphi_N \tilde{V}_N^{\alpha}=(\beta_N-\beta_N^{\alpha})(\varphi_N)^{\frac{1}{1-\alpha}}.
	\end{align*}
	Define $\beta_{N-1}=\beta_N-\beta_N^{\alpha}$. Then we have $\tilde{V}_{N-1}=\beta_{N-1}-\beta_{N-1}^{\alpha}$. Substituting the above value $\tilde{V}_{N-1}$ into \eqref{eq:23}, it can be inferred that 
 \vspace{-0.2cm}
	\begin{align*}
		\tilde{V}_{N-2}=\beta_{N-2}(\varphi_N)^{\frac{1}{1-\alpha}},
	\end{align*}
	where $\beta_{N-2}= \beta_{N-1}-a_{N-1}\beta_{N-1}^{\alpha}$ and $a_{N-1}=\frac{\varphi_{N-1}}{\varphi_N}$.
 
	Similarly, with a recursive relation of $\beta_k$ for $1\leq k \leq N$, $\tilde{V}_{k-1}$ can be expressed as
 \vspace{-0.2cm}
	\begin{align}\label{tilV}
		\tilde{V}_{k-1}=\beta_{k-1}(\varphi_N)^{\frac{1}{1-\alpha}},
	\end{align}
	where $\beta_{k-1}=\beta_{k} - a_{k}\beta_{k}^{\alpha}$ and $a_{k}=\frac{\varphi_{k}}{\varphi_N}$.
	If $\tilde{V}_k$ and $\varphi_k$ satisfy \eqref{range}, then we obtain
 \vspace{-0.2cm}
	\begin{align*}
		\beta_{k-1} \leq & \beta_{k} -(1-\epsilon )\beta_k^{\alpha},\nonumber                  \\
		=                & \epsilon \beta_k^{\alpha} - (1-\beta_k^{1-\alpha})\beta_k^{\alpha}.
	\end{align*}
 
	Since $\tilde{V}_{k-1}=\tilde{V}(K_{k-1}-K^{\star})$ is positive definite, $\beta_{k-1}$ is non-negative. When $\beta_{k-1}=0$, it follows that
 \vspace{-0.2cm}
	\begin{align}\label{ep}
		\epsilon=1-\beta_k^{1-\alpha} \Leftrightarrow \beta_k^{1-\alpha}=1-\epsilon.
	\end{align}
	Let $k=\xi^{\star}$ be the smallest integer for which \eqref{ep} is satisfied, i.e., $\beta_{\xi^{\star}}=(1-\epsilon)^{\frac{1}{1-\alpha}}$. In other words, $\beta_{\xi^{\star}-1}=0$. Thus, it is easy to obtain that $\tilde{V}_{k}=0$ with $\beta_k=0$ for $0\leq k<\xi^{\star}$. Consequently, $K_k$ converges to $K^{\star}$ inversely in finite-time $\xi^{\star}$.
	The proof is completed.

\vspace{-0.2cm}
\subsection{Proof of Theorem \ref{th:convergence}}\label{D}
The proof is completed by utilizing discrete-time Lyapunov analysis.
With Corollary \ref{c1} and Lemma \ref{l3}, we just need to find a Lyapunov function $\tilde{V}_k$, which satisfies the convergence condition in \eqref{eq:23}. The details are shown below.

We define the non-negative Lyapunov function $\tilde{V}_k$ as
\vspace{-0.2cm}
\begin{align}
	\tilde{V}_k= \frac{1}{2} \text{trace}[(K_k-K^{\star})^{\mathrm{T}}(K_k-K^{\star})],
\end{align}
where $K^{\star}$ is the steady-state matrix of inverse convergence. Let $\tilde{K}_k \triangleq K_k-K^{\star}$. Then we have $\tilde{V}_k=\frac{1}{2} \text{trace}[\tilde{K}_k^{\mathrm{T}}\tilde{K}_k]$. Thus, $\tilde{V}_{k-1}-\tilde{V}_{k}$ can be rewritten as
\vspace{-0.2cm}
\begin{align}\label{difference}
	 & ~~~\tilde{V}_{k-1}-\tilde{V}_{k}\nonumber                                                                                                                                                                       \\
	 & = \frac{1}{2} \text{trace} [(K_{k-1}-K^{\star})^{\mathrm{T}}(K_{k-1}-K^{\star})\nonumber                                                                                                                        \\
	 & ~~~-(K_{k}-K^{\star})^{\mathrm{T}}(K_{k}-K^{\star})],\nonumber                                                                                                                                                  \\
	 & = \frac{1}{2} \text{trace} [\tilde{K}_{k-1}^{\mathrm{T}} \tilde{K}_{k-1} - \tilde{K}_{k}^{\mathrm{T}} \tilde{K}_{k}],\nonumber                                                                                  \\
	 & = \frac{1}{2} \text{trace} [(\tilde{K}_{k-1}- \tilde{K}_{k})^{\mathrm{T}} (\tilde{K}_{k-1}+ \tilde{K}_{k})],\nonumber                                                                                           \\
	 & =-\frac{\frac{1}{2} \text{trace} [(\tilde{K}_{k}- \tilde{K}_{k-1})^{\mathrm{T}} (\tilde{K}_{k-1}+ \tilde{K}_{k})]}{\{\frac{1}{2}\text{trace}[\tilde{K}_k^{\mathrm{T}}\tilde{K}_k]\}^{1/p}} (\tilde{V}_k)^{1/p}.
\end{align}
Let $\frac{1}{2}<1/p <1$. To apply Lemma \ref{l3}, we need to find the upper bound of the term on the right side of equality \eqref{difference}.

Consider the term,
\vspace{-0.2cm}
\begin{align}\label{eq34}
	                      & \frac{\frac{1}{2} \text{trace} [(\tilde{K}_{k-1}- \tilde{K}_{k})^{\mathrm{T}} (\tilde{K}_{k-1}+ \tilde{K}_{k})]}{\{\frac{1}{2}\text{trace}[\tilde{K}_k^{\mathrm{T}}\tilde{K}_k]\}^{1/p}} \nonumber                                                               \\
	\overset{(s.1)}{\leq} & \frac{\frac{1}{2}\sqrt{\text{trace}[(\tilde{K}_{k-1}^{\mathrm{T}}-\tilde{K}_k^{\mathrm{T}})^2]}\sqrt{\text{trace}[(\tilde{K}_{k-1}^{\mathrm{T}}+\tilde{K}_k^{\mathrm{T}})^2]}}{\{\frac{1}{2}\text{trace}[\tilde{K}_k^{\mathrm{T}}\tilde{K}_k]\}^{1/p}},\nonumber \\
	\overset{(s.2)}{=}    & \{\tilde{V_k}+\tilde{V}_{k-1}-\text{trace}[\tilde{K}_k^{\mathrm{T}}\tilde{K}_{k-1}]\}^{\frac{1}{2}}\nonumber                                                                                                                   \\
	                      & \frac{\{\tilde{V_k}+\tilde{V}_{k-1}+\text{trace}[\tilde{K}_k^{\mathrm{T}}\tilde{K}_{k-1}]\}^{\frac{1}{2}}}{(\tilde{V}_k)^{1/p}},
\end{align}
where $(s.1)$ follows the fact that $\text{trace}[A^{\mathrm{T}}B]\leq \{\text{trace}[A^{\mathrm{T}}A]\}^{\frac{1}{2}} \{[\text{trace}[B^{\mathrm{T}}B]\}^{\frac{1}{2}}$ for any $n$-order real symmetric matrix, $(s.2)$ follows the fact that $\text{trace}[\tilde{K}_k^{\mathrm{T}}\tilde{K}_k]=2 \tilde{V}_k$ and $\text{trace}[A]=\text{trace}[A^{\mathrm{T}}]$.
Since $\tilde{V}_{k-1}= \beta_{k-1}(\varphi_N)^{\frac{1}{1-\alpha}}$ in \eqref{tilV} and $\tilde{V}_{k}=\beta_{k}(\varphi_N)^{\frac{1}{1-\alpha}}$, one gets
\vspace{-0.1cm}
\begin{align}
	\tilde{V}_{k-1}=\eta_k \tilde{V}_k,
\end{align}
with $\eta_k=\frac{\beta_{k-1}}{\beta_{k}}=\frac{\text{trace}[\tilde{K}_{k-1}^{\mathrm{T}}\tilde{K}_{k-1}]}{\text{trace}[\tilde{K}_k^{\mathrm{T}}\tilde{K}_k]}$. In addition, with $\zeta_k=\frac{\text{trace}[\tilde{K}_k^{\mathrm{T}}\tilde{K}_{k-1}]}{\text{trace}[\tilde{K}_k^{\mathrm{T}}\tilde{K}_{k}]}$, we have
\vspace{-0.2cm}
\begin{align}\label{eq37}
	\text{trace}[\tilde{K}_k^{\mathrm{T}}\tilde{K}_{k-1}]=\zeta_k \tilde{V}_k.
\end{align}
Thus, \eqref{eq34} can be rewritten as
\vspace{-0.2cm}
\begin{align}\label{eq38}
	 & \frac{\frac{1}{2} \text{trace} [(\tilde{K}_{k-1}- \tilde{K}_{k})^{\mathrm{T}} (\tilde{K}_{k-1}+ \tilde{K}_{k})]}{\{\frac{1}{2}\text{trace}[\tilde{K}_k^{\mathrm{T}}\tilde{K}_k]\}^{1/p}}\nonumber \\
	 & \leq \sqrt{(1+\eta_k)^2-(\zeta_k)^2} (\tilde{V}_k)^{1-1/p}.
\end{align}
Since $\tilde{K}_k$ and $\tilde{K}_{k-1}$ are $n$-order real symmetric matrix from Lemma \ref{l1}, it follows that
\begin{align}\label{eq39}
	\text{trace}[\tilde{K}_k^{\mathrm{T}}\tilde{K}_{k-1}] \leq & \left\{\text{trace}[\tilde{K}_k^{\mathrm{T}}\tilde{K}_{k}]\right\}^{\frac{1}{2}} \left\{\text{trace}[\tilde{K}_{k-1}^{\mathrm{T}}\tilde{K}_{k-1}]\right\}^{\frac{1}{2}},\nonumber \\
	\leq                                                       & \sqrt{2\tilde{V}_k}\sqrt{2\tilde{V}_{k-1}}.
\end{align}
Combined \eqref{eq37} with \eqref{eq39}, we have
\vspace{-0.2cm}
\begin{align*}
	\zeta_k\leq 2\sqrt{\frac{\tilde{V}_{k-1}}{\tilde{V}_k}}=2\sqrt{\eta_k}.
\end{align*}
Hence, \eqref{eq38} can be rewritten as
\vspace{-0.2cm}
\begin{align*}
\frac{\frac{1}{2} \text{trace} [(\tilde{K}_{k-1}- \tilde{K}_{k})^{\mathrm{T}} (\tilde{K}_{k-1}+ \tilde{K}_{k})]}{\{\frac{1}{2}\text{trace}[\tilde{K}_k^{\mathrm{T}}\tilde{K}_k]\}^{1/p}}\leq |\eta_k-1| (\tilde{V}_k)^{1-1/p}.
\end{align*}
One can see that $\tilde{V}_{k-1}-\tilde{V}_{k}=0$ if $\eta_k=1$ or $\tilde{V}_{k}=0$. Furthermore, if and only if $\tilde{V}_{k}=0$, we conclude that
\vspace{-0.2cm}
\begin{align*}
	\tilde{V}_{k-1}-\tilde{V}_{k}=0 \Leftrightarrow \tilde{K}_k=0 \Leftrightarrow K_k=K^{\star}.
\end{align*}
It shows that $\tilde{V}_{k}=0$ if and only if $K_k$ is at the equilibrium matrix $K^{\star}$.

We consider $\tilde{K}_{k-1}=\mu_k \tilde{K}_k$ with $\mu_k>0$. To guarantee that \eqref{eq:23} holds, combined with \eqref{difference}, we have
\vspace{-0.2cm}
\begin{align*}
	\tilde{V}_{k-1}-\tilde{V}_{k}= & (\mu_k^2-1)\tilde{V}_{k},\nonumber                       \\
	=                              & -[(1-\mu_k^2)\tilde{V}_{k}^{1-1/p}] \tilde{V}_{k}^{1/p},
\end{align*}
where $(1-\mu_k^2)\tilde{V}_{k}^{1-1/p}$ is the function of $\tilde{V}_{k}^{1-\alpha}$ with $\alpha=1/p$. If $\mu_k$ satisfies
\vspace{-0.2cm}
\begin{align*}
	\mu_k=\frac{(2\tilde{V}_{k})^{1-1/p}-\kappa}{(2\tilde{V}_{k})^{1-1/p}+\kappa},
\end{align*}
\vspace{-0.2cm}
where $\kappa>0$, then we have
\vspace{-0.2cm}
\begin{align*}
	\varphi_k \triangleq (1-\mu_k^2)\tilde{V}_{k}^{1-1/p}=4\kappa \frac{2^{1-1/p}(\tilde{V}_k)^{2-2/p}}{[(2\tilde{V}_k)^{1-1/p}+\kappa]^2}.
\end{align*}
From \eqref{range}, $\tilde{V}_k$ is decreasing if $0<\varphi_k<\frac{4\kappa}{2^{1-1/p}}$ for $\tilde{V}_{k}^{1-1/p}\in(0,\tilde{V}_{N}^{1-1/p})$. In addition, the ratio $\frac{\varphi_k}{\varphi_N}$ is bounded and higher than a positive constant in the open interval $(0,1)$ since
\vspace{-0.2cm}
\begin{align*}
	\frac{\varphi_k}{\varphi_N}=\left(\frac{\tilde{V}_k}{\tilde{V}_N}\right)^{2-2/p} \frac{[(2\tilde{V}_N)^{1-1/p}+\kappa]^2}{[(2\tilde{V}_k)^{1-1/p}+\kappa]^2}.
\end{align*}

Therefore, $K_k$ will converge to $K^{\star}$ inversely for $0\leq k< \xi^{\star}$ where the positive integer $\xi^{\star}$ satisfies \eqref{ep}.
Finally, the proof is completed.

\addtolength{\textheight}{-12cm}   







\bibliographystyle{IEEEtran}
\bibliography{reference}

\begin{thebibliography}{10}
\providecommand{\url}[1]{#1}
\csname url@samestyle\endcsname
\providecommand{\newblock}{\relax}
\providecommand{\bibinfo}[2]{#2}
\providecommand{\BIBentrySTDinterwordspacing}{\spaceskip=0pt\relax}
\providecommand{\BIBentryALTinterwordstretchfactor}{4}
\providecommand{\BIBentryALTinterwordspacing}{\spaceskip=\fontdimen2\font plus
\BIBentryALTinterwordstretchfactor\fontdimen3\font minus
  \fontdimen4\font\relax}
\providecommand{\BIBforeignlanguage}[2]{{%
\expandafter\ifx\csname l@#1\endcsname\relax
\typeout{** WARNING: IEEEtran.bst: No hyphenation pattern has been}%
\typeout{** loaded for the language `#1'. Using the pattern for}%
\typeout{** the default language instead.}%
\else
\language=\csname l@#1\endcsname
\fi
#2}}
\providecommand{\BIBdecl}{\relax}
\BIBdecl

\bibitem{zhang2019networked}
X.-M. Zhang, Q.-L. Han, X.~Ge, D.~Ding, L.~Ding, D.~Yue, and C.~Peng,
  ``Networked control systems: A survey of trends and techniques,''
  \emph{IEEE/CAA JAS}, vol.~7, no.~1, pp. 1--17, 2019.

\bibitem{luo2022feedback}
X.~Luo, C.~Fang, J.~He, C.~Zhao, and D.~Paccagnan, ``A
  feedback-optimization-based model-free attack scheme in networked control
  systems,'' \emph{arXiv preprint arXiv:2212.07633}, 2022.

\bibitem{Gonjeshke2022attack}
G.~Darande, ``cyberattacks against iran's steel industry,''
  \url{https://www.youtube.com/watch?v=fnbCr8mgoT8}, 2022.

\bibitem{liu2011false}
Y.~Liu, P.~Ning, and M.~K. Reiter, ``False data injection attacks against state
  estimation in electric power grids,'' \emph{ACM Trans. Inf. Syst. Secur.},
  vol.~14, no.~1, pp. 1--33, 2011.

\bibitem{mo2010false}
Y.~Mo and B.~Sinopoli, ``False data injection attacks in control systems,'' in
  \emph{Workshop on Secure Control Syst.}, 2010, pp. 1--6.

\bibitem{sui2020vulnerability}
T.~Sui, Y.~Mo, D.~Marelli, X.~Sun, and M.~Fu, ``The vulnerability of
  cyber-physical system under stealthy attacks,'' \emph{IEEE TAC}, vol.~66,
  no.~2, pp. 637--650, 2020.

\bibitem{zhu2014resilience}
Y.~Zhu, J.~Yan, Y.~Tang, Y.~L. Sun, and H.~He, ``Resilience analysis of power
  grids under the sequential attack,'' \emph{IEEE TIFS}, vol.~9, no.~12, pp.
  2340--2354, 2014.

\bibitem{tan2017modeling}
R.~Tan, H.~H. Nguyen, E.~Y. Foo, D.~K. Yau, Z.~Kalbarczyk, R.~K. Iyer, and
  H.~B. Gooi, ``Modeling and mitigating impact of false data injection attacks
  on automatic generation control,'' \emph{IEEE TIFS}, vol.~12, no.~7, pp.
  1609--1624, 2017.

\bibitem{chen2017cyber}
Y.~Chen, S.~Kar, and J.~M. Moura, ``Cyber-physical attacks with control
  objectives,'' \emph{IEEE TAC}, vol.~63, no.~5, pp. 1418--1425, 2017.

\bibitem{li2019optimal}
Y.-G. Li and G.-H. Yang, ``Optimal stealthy false data injection attacks in
  cyber-physical systems,'' \emph{Inf. Sciences}, vol. 481, pp. 474--490, 2019.

\bibitem{jafari2022optimal}
M.~Jafari, M.~A. Rahman, and S.~Paudyal, ``Optimal false data injection attacks
  against power system frequency stability,'' \emph{IEEE TSG}, 2022.

\bibitem{wu2018optimal}
G.~Wu, J.~Sun, and J.~Chen, ``Optimal data injection attacks in cyber-physical
  systems,'' \emph{IEEE TCYB}, vol.~48, no.~12, pp. 3302--3312, 2018.

\bibitem{wu2019optimal}
G.~Wu, G.~Wang, J.~Sun, and L.~Xiong, ``Optimal switching attacks and
  countermeasures in cyber-physical systems,'' \emph{IEEE TSMC}, vol.~51,
  no.~8, pp. 4825--4835, 2019.

\bibitem{ye2020complexity}
L.~Ye, N.~Woodford, S.~Roy, and S.~Sundaram, ``On the complexity and
  approximability of optimal sensor selection and attack for kalman
  filtering,'' \emph{IEEE TAC}, vol.~66, no.~5, pp. 2146--2161, 2020.

\bibitem{luo2022submodularity}
X.~Luo, C.~Zhao, C.~Fang, and J.~He, ``Submodularity-based false data injection
  attack scheme in multi-agent dynamical systems,'' in \emph{IEEE ACC}, 2022,
  pp. 4998--5003.

\bibitem{luo2023optimal}
X.~Luo, C.~Fang, C.~Zhao, P.~Cheng, and J.~He, ``Optimal sequential false data
  injection attack scheme: Finite-time inverse convergence,'' in \emph{IEEE
  CDC}, 2023.

\bibitem{Zhang2020optimal}
Q.~Zhang, K.~Liu, Y.~Xia, and A.~Ma, ``Optimal stealthy deception attack
  against cyber-physical systems,'' \emph{IEEE TCYB}, vol.~50, no.~9, pp.
  3963--3972, 2020.

\bibitem{Lu2022false}
A.-Y. Lu and G.-H. Yang, ``False data injection attacks against state
  estimation without knowledge of estimators,'' \emph{IEEE TAC}, vol.~67,
  no.~9, pp. 4529--4540, 2022.

\bibitem{luo2022model}
X.~Luo, C.~Fang, C.~Zhao, and J.~He, ``A model-free false data injection attack
  strategy in networked control systems,'' in \emph{IEEE CDC}, 2022, pp.
  2941--2946.

\bibitem{kim2014subspace}
J.~Kim, L.~Tong, and R.~J. Thomas, ``Subspace methods for data attack on state
  estimation: A data driven approach,'' \emph{IEEE TSP}, vol.~63, no.~5, pp.
  1102--1114, 2014.

\bibitem{an2017data}
L.~An and G.-H. Yang, ``Data-driven coordinated attack policy design based on
  adaptive $\mathcal{L}_2$-gain optimal theory,'' \emph{IEEE TAC}, vol.~63,
  no.~6, pp. 1850--1857, 2017.

\bibitem{zhao2022data}
Z.~Zhao, Y.~Xu, Y.~Li, Z.~Zhen, Y.~Yang, and Y.~Shi, ``Data-driven attack
  detection and identification for cyber-physical systems under sparse sensor
  attacks,'' \emph{IEEE TAC}, 2022.

\bibitem{Mousavinejad2021resilient}
E.~Mousavinejad, X.~Ge, Q.-L. Han, F.~Yang, and L.~Vlacic, ``Resilient tracking
  control of networked control systems under cyber attacks,'' \emph{IEEE TCYB},
  vol.~51, no.~4, pp. 2107--2119, 2021.

\bibitem{wang2021optimal}
X.-L. Wang, ``Optimal attack strategy against fault detectors for linear
  cyber-physical systems,'' \emph{Information Sciences}, vol. 581, pp.
  390--402, 2021.

\bibitem{Pasqualetti2013attack}
F.~Pasqualetti, F.~Dörfler, and F.~Bullo, ``Attack detection and
  identification in cyber-physical systems,'' \emph{IEEE TAC}, vol.~58, no.~11,
  pp. 2715--2729, 2013.

\bibitem{guo2018worst}
Z.~Guo, D.~Shi, K.~H. Johansson, and L.~Shi, ``Worst-case stealthy
  innovation-based linear attack on remote state estimation,''
  \emph{Automatica}, vol.~89, pp. 117--124, 2018.

\bibitem{hamrah2019discrete}
R.~Hamrah, A.~K. Sanya, and S.~P. Viswanathan, ``Discrete finite-time stable
  position tracking control of unmanned vehicles,'' in \emph{IEEE CDC}, 2019,
  pp. 7025--7030.

\bibitem{puterman1990markov}
M.~L. Puterman, ``Markov decision processes,'' \emph{Handbooks in operations
  research and management science}, vol.~2, pp. 331--434, 1990.

\bibitem{schulman2017proximal}
J.~Schulman, F.~Wolski, P.~Dhariwal, A.~Radford, and O.~Klimov, ``Proximal
  policy optimization algorithms,'' \emph{arXiv preprint arXiv:1707.06347},
  2017.

\bibitem{sutton2018reinforcement}
R.~S. Sutton and A.~G. Barto, \emph{Reinforcement learning: An
  introduction}.\hskip 1em plus 0.5em minus 0.4em\relax MIT press, 2018.

\bibitem{lillicrap2015continuous}
T.~P. Lillicrap, J.~J. Hunt, A.~Pritzel, N.~Heess, T.~Erez, Y.~Tassa,
  D.~Silver, and D.~Wierstra, ``Continuous control with deep reinforcement
  learning,'' \emph{arXiv preprint arXiv:1509.02971}, 2015.

\bibitem{schulman2015trust}
J.~Schulman, S.~Levine, P.~Abbeel, M.~Jordan, and P.~Moritz, ``Trust region
  policy optimization,'' in \emph{ICML}, 2015, pp. 1889--1897.

\bibitem{agarwal2021theory}
A.~Agarwal, S.~M. Kakade, J.~D. Lee, and G.~Mahajan, ``On the theory of policy
  gradient methods: Optimality, approximation, and distribution shift,''
  \emph{JMLR}, vol.~22, no.~1, pp. 4431--4506, 2021.

\bibitem{wang2019neural}
L.~Wang, Q.~Cai, Z.~Yang, and Z.~Wang, ``Neural policy gradient methods: Global
  optimality and rates of convergence,'' \emph{arXiv preprint
  arXiv:1909.01150}, 2019.

\end{thebibliography}

\end{document}